\documentclass [a4paper,10pt]{article}
\usepackage{arcs}
\usepackage[margin=1in]{geometry}

\usepackage{amsthm}
\usepackage[utf8]{inputenc}
\usepackage{textcomp}
\usepackage{dsfont}
\usepackage{latexsym}
\usepackage{amssymb}
\usepackage{amsmath}
\usepackage{mathtools}
\usepackage{xfrac}

\usepackage{etoolbox}
\apptocmd{\lim}{\limits}{}{}
\apptocmd{\sup}{\limits}{}{}
\apptocmd{\inf}{\limits}{}{}
\apptocmd{\liminf}{\limits}{}{}
\apptocmd{\limsup}{\limits}{}{}
\pretocmd{\langle}{\left}{}{}
\pretocmd{\rangle}{\right}{}{}

\DeclareMathOperator{\pr}{pr}

\newcommand{\vect}[1]{\mathbf{{#1}}}

\newcommand{\leftexp}[2]{{\vphantom{#2}}^{#1}{#2}} 
\newcommand{\trasp}{\leftexp{t}}

\newcommand{\Cc}{\mathcal{C}}
\newcommand{\Dc}{\mathcal{D}}
\newcommand{\Ec}{\mathcal{E}}
\newcommand{\Fc}{\mathcal{F}}
\newcommand{\Hc}{\mathcal{H}}
\newcommand{\Kc}{\mathcal{K}}
\newcommand{\Lc}{\mathcal{L}}
\newcommand{\Mc}{\mathcal{M}}
\newcommand{\Sc}{\mathcal{S}}
\newcommand{\Tc}{\mathcal{T}}

\newcommand{\Sd}{\mathds{S}}

\newcommand{\mi}{\mu}
\renewcommand{\phi}{\varphi}
\newcommand{\eps}{\varepsilon}
\newcommand{\id}{\mathrm{id}}

\newcommand{\N}{\mathds{N}}
\newcommand{\Z}{\mathds{Z}}
\newcommand{\Q}{\mathds{Q}}
\newcommand{\R}{\mathds{R}}
\newcommand{\C}{\mathds{C}}
\newcommand{\Hd}{\mathds{H}}

\newcommand{\gf}{\mathfrak{g}}

\newcommand{\vf}{\mathfrak{v}}

\newcommand{\Ff}{\mathfrak{F}}

\newcommand{\Uf}{\mathfrak{U}}

\newcommand{\dd}{\mathrm{d}}
\newcommand{\loc}{\mathrm{loc}}

\DeclareMathOperator{\card}{Card}

\DeclarePairedDelimiter{\abs}{\lvert}{\rvert}
\DeclarePairedDelimiter{\norm}{\lVert}{\rVert}
\DeclarePairedDelimiterX\Set[1]\lbrace\rbrace{#1}

\newcommand{\Meg}{\geqslant}
\newcommand{\meg}{\leqslant}
\newcommand{\Supp}[1]{\mathrm{Supp}\left( #1\right)}

\newcommand{\quot}[2]{\mathchoice
	{\setbox1\hbox{${\displaystyle #1}_{\scriptstyle #2}$}
		\text{$#1$\raisebox{-1mm}{\resizebox{2.8mm}{3.5mm}{$/$}}\raisebox{-1.8mm}{$#2$}}}
	{\setbox1\hbox{${\textstyle #1}_{\scriptstyle #2}$}
		\text{$#1$\raisebox{-1mm}{\resizebox{2.8mm}{3.5mm}{$/$}}\raisebox{-1.8mm}{$#2$}}}
	{\setbox1\hbox{${\scriptstyle #1}_{\scriptscriptstyle #2}$}
		\text{$#1$\raisebox{-0.6mm}{\resizebox{2.3mm}{2.0125mm}{$/$}}\raisebox{-1.4mm}{$ #2$}}}
	{\setbox1\hbox{${\scriptscriptstyle #1}_{\scriptscriptstyle #2}$}
		\text{$#1$\raisebox{-0.3mm}{\resizebox{1.55mm}{1.4mm}{$/$}}\raisebox{-0.7mm}{$#2$}}}}

\DeclareMathOperator{\Pfaff}{Pf}

\newcommand{\open}{\overset{\resizebox{1.1mm}{1.1mm}{$\circ$}}}
\DeclareMathOperator{\tr}{Tr}

\makeatletter
\newcommand*{\mint}[1]{%
	\mint@l{#1}{}%
}
\newcommand*{\mint@l}[2]{%
	\@ifnextchar\limits{%
		\mint@l{#1}%
	}{%
	\@ifnextchar\nolimits{%
		\mint@l{#1}%
	}{%
	\@ifnextchar\displaylimits{%
		\mint@l{#1}%
	}{%
	\mint@s{#2}{#1}%
}%
}%
}%
}
\newcommand*{\mint@s}[2]{%
	\@ifnextchar_{%
		\mint@sub{#1}{#2}%
	}{%
	\@ifnextchar^{%
		\mint@sup{#1}{#2}%
	}{%
	\mint@{#1}{#2}{}{}%
}%
}%
}
\def\mint@sub#1#2_#3{%
	\@ifnextchar^{%
		\mint@sub@sup{#1}{#2}{#3}%
	}{%
	\mint@{#1}{#2}{#3}{}%
}%
}
\def\mint@sup#1#2^#3{%
	\@ifnextchar_{%
		\mint@sub@sup{#1}{#2}{#3}%
	}{%
	\mint@{#1}{#2}{}{#3}%
}%
}
\def\mint@sub@sup#1#2#3^#4{%
	\mint@{#1}{#2}{#3}{#4}%
}
\def\mint@sup@sub#1#2#3_#4{%
	\mint@{#1}{#2}{#4}{#3}%
}
\newcommand*{\mint@}[4]{
	\mathop{}%
	\mkern-\thinmuskip
	\mathchoice{%
		\mint@@{#1}{#2}{#3}{#4}%
		\displaystyle\textstyle\scriptstyle
	}{%
	\mint@@{#1}{#2}{#3}{#4}%
	\textstyle\scriptstyle\scriptstyle
}{%
\mint@@{#1}{#2}{#3}{#4}%
\scriptstyle\scriptscriptstyle\scriptscriptstyle
}{%
\mint@@{#1}{#2}{#3}{#4}%
\scriptscriptstyle\scriptscriptstyle\scriptscriptstyle
}%
\mkern-\thinmuskip
\int#1%
\ifx\\#3\\\else_{#3}\fi
\ifx\\#4\\\else^{#4}\fi  
}

\newcommand*{\mint@@}[7]{
	\begingroup
	\sbox0{$#5\int\m@th$}%
	\sbox2{$#5\int_{}\m@th$}%
	\dimen2=\wd0 %
	\let\mint@limits=#1\relax
	\ifx\mint@limits\relax
	\sbox4{$#5\int_{\kern1sp}^{\kern1sp}\m@th$}%
	\ifdim\wd4>\wd2 %
	\let\mint@limits=\nolimits
	\else
	\let\mint@limits=\limits
	\fi
	\fi
	\ifx\mint@limits\displaylimits
	\ifx#5\displaystyle
	\let\mint@limits=\limits
	\fi
	\fi
	\ifx\mint@limits\limits
	\sbox0{$#7#3\m@th$}%
	\sbox2{$#7#4\m@th$}%
	\ifdim\wd0>\dimen2 %
	\dimen2=\wd0 %
	\fi
	\ifdim\wd2>\dimen2 %
	\dimen2=\wd2 %
	\fi
	\fi
	\rlap{%
		$#5%
		\vcenter{%
			\hbox to\dimen2{%
				\hss
				$#6{#2}\m@th$%
				\hss
			}%
		}%
		$%
	}%
	\endgroup
}

\begin{document}
\title{Spectral Multipliers on $2$-Step Stratified Groups, II}
\author{Mattia Calzi\thanks{The author is partially supported by the grant PRIN 2015 \emph{Real and Complex Manifolds:  Geometry, Topology and Harmonic Analysis}, and is member of the Gruppo Nazionale per l’Analisi Matematica, la Probabilità e le loro Applicazioni (GNAMPA) of the Istituto Nazionale di Alta Matematica (INdAM).}}
\date{}

\theoremstyle{definition}
\newtheorem{deff}{Definition}[section]

\newtheorem{oss}[deff]{Remark}

\newtheorem{ass}[deff]{Assumptions}

\newtheorem{nott}[deff]{Notation}

\theoremstyle{plain}
\newtheorem{teo}[deff]{Theorem}

\newtheorem{lem}[deff]{Lemma}

\newtheorem{prop}[deff]{Proposition}

\newtheorem{cor}[deff]{Corollary}

\maketitle
\begin{small}
	\section*{Abstract}
	Given a graded group $G$ and commuting, formally self-adjoint, left-invariant, homogeneous differential operators $\Lc_1,\dots, \Lc_n$ on $G$, one of which is Rockland, we study the convolution operators $m(\Lc_1,\dots, \Lc_n)$ and their convolution kernels, with particular reference to the case in which $G$ is abelian and $n=1$, and the case in which $G$ is a $2$-step stratified group which satisfies a slight strengthening of the Moore-Wolf condition and $\Lc_1,\dots,\Lc_n$ are either sub-Laplacians or central elements of the Lie algebra of $G$. 
	Under suitable conditions, we prove that: i) if the convolution kernel of the operator $m(\Lc_1,\dots, \Lc_n)$ belongs to $L^1$, then $m$ equals almost everywhere a continuous function vanishing at $\infty$ (`Riemann-Lebesgue lemma'); ii) if the convolution kernel of the operator $m(\Lc_1,\dots, \Lc_n)$ is a Schwartz function, then $m$ equals almost everywhere a Schwartz function.
\end{small}

\section{Introduction}

Given a Rockland family\footnote{see Section~\ref{sec:2} for precise definitions} $(\Lc_1,\dots, \Lc_n)$ on a homogeneous group $G$, following~\cite{Martini,Tolomeo} (see also~\cite{Calzi}) we define a `kernel transform' $\Kc$ which to every measurable function $m\colon \R^n\to \C$ such that $m(\Lc_1,\dots, \Lc_n)$ is defined on $C^\infty_c(G)$ associates a unique distribution $\Kc(m)$ such that
\[
m(\Lc_1,\dots, \Lc_n)\,\phi=\phi*\Kc(m)
\]
for every $\phi \in C^\infty_c(G)$. 
The so-defined kernel transform $\Kc$ enjoys some relevant properties, which we list below; see~\cite{Martini,Tolomeo} for their proofs and further information. 
\begin{itemize}
	\item there is a unique positive Radon measure $\beta$ on $\R^n$ such that $\Kc(m)\in L^2(G)$ if and only if $m\in L^2(\beta)$, and $\Kc$ induces an isometry of $L^2(\beta)$ into $L^2(G)$;
	
	\item there is a unique $\chi\in L^\infty(\R^n\times G, \beta \otimes \nu)$, where $\nu$ denotes a Haar measure on $G$, such that for every $m\in L^1(\beta)$
	\[
	\Kc(m)(g)=\int_{\R^n} m(\lambda) \chi(\lambda, g)\,\dd \beta(\lambda)
	\] 
	for almost every $g\in G$;
	
	\item $\Kc$ maps $\Sc(\R^n)$ into $\Sc(G)$.
\end{itemize}

We consider also some additional properties of particular interest, such as:
\begin{itemize}
	\item[$(RL)$] if $\Kc(m)\in L^1(G)$, then we can take $m$ so as to belong to $C_0(\R^n)$;
	
	\item[$(S)$] if $\Kc(m)\in \Sc(G)$, then we can take $m$ so as to belong to $\Sc(\R^n)$.
\end{itemize}

In this paper, we shall investigate the validity of properties $(RL)$ and $(S)$ in two particular cases: that of a Rockland operator on an abelian group, and that of homogeneous sub-Laplacians and elements of the centre on an $MW^+$ group (cf.~Definition~\ref{def:3:1}).

Here is a plan of the following sections.
In Section~\ref{sec:2} we recall the basic definitions and notation, as well as some relevant results proved in~\cite{Calzi}.
In Section~\ref{sec:3}, we then consider abelian groups, and characterize the Rockland operators which satisfy property $(S)$ thereon.
In Section~\ref{sec:4} we prepare the machinery for the study of homogeneous sub-Laplacians and elements of the centre on $MW^+$ groups,  referring to~\cite{Calzi} for the proof of analogous statements when necessary. 
In contrast with the situation considered in~\cite{Calzi}, the structure of $MW^+$ groups will allow us to treat more than one homogeneous sub-Laplacian at a time.
In Sections~\ref{sec:5} and~\ref{sec:6}, then, we prove some sufficient conditions for properties $(RL)$ and $(S)$ in this context.

In Section~\ref{sec:7} we present a particularly elegant result where all the  good properties we consider are proved to be equivalent for the families which are invariant (in some sense) under the action of suitable groups of isometries.
In particular, this result covers the case of Heisenberg groups, thanks to the results of Section~\ref{sec:3}. 
Finally, in Section~\ref{sec:8} we consider products of Heisenberg groups and `decomposable' homogeneous sub-Laplacian thereon. In addition, we exhibit a Rockland family which is `functionally complete' (cf.~Definition~\ref{def:1:1}) but does not satisfy property $(S)$.

\section{Preliminaries}\label{sec:2}

In this section we recall some basic results and definitions from~\cite{Calzi}.  We shall then prove some useful results that were not considered therein.

\subsection{General Definitions and Notation}

As in~\cite{Calzi}, a Rockland family on a homogeneous group $G$ (cf.~\cite{FollandStein}) is a  jointly hypoelliptic, commutative, finite family $\Lc_A=(\Lc_\alpha)_{\alpha\in A}$ of formally self-adjoint, homogeneous, left-invariant differential operators without constant terms.
In this case, the $\Lc_\alpha$ are essentially self-adjoint on $C^\infty_c(G)$, and their closures commute. 
In addition, $\Lc_A$ is a weighted subcoercive system of operators (cf.~\cite[Proposition 3.6.3]{Martini}), so that the theory developed in~\cite{Martini} applies.

\begin{deff}\label{def:1}
To every (Borel, say) measurable function $m\colon \R^A\to \C$ such that $m(\Lc_A)$ is defined (at least) on $C^\infty_c(G)$, we associate a unique distribution $\Kc_{\Lc_A}(m)$ (its `kernel') on $G$ such that
\[
m(\Lc_A)(\phi)=\phi* \Kc_{\Lc_A}(m)
\]
for every $\phi \in C^\infty_c(G)$.
\end{deff}

We denote by $E_{\Lc_A}$ the space $\R^A$ endowed with the dilations defined by
\[
r\cdot (\lambda_\alpha)\coloneqq  (r^{\delta_\alpha} \lambda_\alpha)
\]
for every $r>0$ and for every $(\lambda_\alpha)\in \R^A$, where $\delta_\alpha$ is the homogeneous degree of $\Lc_\alpha$. 
We shall often employ the following short-hand notation: $L^1_{\Lc_A}(G)$ and $\Sc_{\Lc_A}(G)$ will denote $\Kc_{\Lc_A}(L^\infty(\beta))\cap L^1(G)$ and $\Kc_{\Lc_A}(L^\infty(\beta))\cap \Sc(G)$, respectively, while $\Sc(G,\Lc_A)$ will denote $\Kc_{\Lc_A}(\Sc(E_{\Lc_A}))$.

Now, by~\cite[Theorem 3.2.7]{Martini} there is a unique positive Radon measure  $\beta_{\Lc_A}$ on $E_{\Lc_A}$ such that a Borel function $m\colon E_{\Lc_A}\to \C$ is square-integrable if and only if $\Kc_{\Lc_A}(m)\in L^2(G)$ and such that, in this case,
\[
\norm{m}_{L^2(\beta_{\Lc_A})}=\norm{\Kc_{\Lc_A}(m)}_{L^2(G)}.
\]
The measure $\beta_{\Lc_A}$ is then equivalent to the spectral measure associated with $\Lc_A$.
Using the existence of $\beta_{\Lc_A}$ and the fact that $\Kc_{\Lc_A}$ maps $\Sc(E_{\Lc_A})$ in $\Sc(G)$, it is not hard to prove that a $\beta_{\Lc_A}$-measurable function admits a kernel in the sense of Definition~\ref{def:1} if and only if there is a positive polynomial $P$ on $E_{\Lc_A}$ such that $\frac{m}{1+P}\in L^2(\beta_{\Lc_A})$.

Now, $\Kc_{\Lc_A}$ can be extended to a continuous linear mapping from $L^1(\beta_{\Lc_A})$ into $C_0(G)$ (cf.~\cite[Proposition 3.2.12]{Martini}), and there is a unique $\chi_{\Lc_A}\in L^\infty(\beta_{\Lc_A}\otimes \nu_G)$, where $\nu_G$ denotes a fixed Haar measure on $G$, such that
\[
\Kc_{\Lc_A}(m)(g)=\int_{E_{\Lc_A}} m(\lambda)\chi_{\Lc_A}(\lambda,g)\,\dd \beta_{\Lc_A}(\lambda)
\]
for every $m\in L^1(\beta_{\Lc_A})$ and for almost every $g\in G$.\footnote{This is basically a consequence of the Dunford--Pettis theorem, cf.~\cite{Tolomeo}.}

Further, we denote by $\Mc_{\Lc_A}\colon L^1(G)\to L^\infty(G)$ the transpose of the mapping $m\mapsto \Kc_{\Lc_A}(m)\check{\;}$, so that
\[
\Mc_{\Lc_A}(f)(\lambda)= \int_{G} f(g) \overline{\chi_{\Lc_A}(\lambda,g)}\,\dd g
\]
for every $f\in L^1(G)$ and for $\beta_{\Lc_A}$-almost every $\lambda\in E_{\Lc_A}$. Observing that $\Mc_{\Lc_A}$ equals the adjoint of the isometry $\Kc_{\Lc_A}\colon L^2(\beta_{\Lc_A})\to L^2(G)$ on $L^1(G)\cap L^2(G)$, one  may then prove that $\Kc_{\Lc_A}\circ\Mc_{\Lc_A}$ is the identity on $L^1_{\Lc_A}(G)$.

\subsection{Products}

Assume that we have two Rockland families $\Lc_A$ and $\Lc'_{A'}$ on two homogeneous groups $G$ and $G'$, respectively. 
Denote by $\Lc''_{A''}$ the family whose elements are the operators on $G\times G'$ induced by the elements of $\Lc_A$ and $\Lc'_{A'}$, and observe that $\Lc''_{A''}$ is a Rockland family.

\begin{teo}[\cite{Calzi}, Theorems 4.5 and 4.10]\label{teo:4}\label{teo:5}
	The families $\Lc_A$ and $\Lc'_{A'}$ satisfy property $(RL)$ (resp.\ $(S)$) if and only if $\Lc''_{A''}$ does.
\end{teo}

\subsection{Composite Functions}

Assume that we are given a Rockland family $\Lc_A$ and a polynomial mapping $P$ on $E_{\Lc_A}$ such that $P(\Lc_A)$ is still a Rockland family (this is equivalent to saying that $P$ is proper and has homogeneous components \emph{with respect to the dilations of $E_{\Lc_A}$}). Then, for every bounded measurable multiplier $m$ we have $\Kc_{P(\Lc_A)}(m)=\Kc_{\Lc_A}(m\circ P)$. 
As a consequence, if we want to establish properties $(RL)$ or $(S)$ for $P(\Lc_A)$ on the base of our knowledge of $\Lc_A$, it is of importance to infer some properties of $m$ from the properties of $m\circ P$.
The results of this section address this problem.

We begin with a definition.

\begin{deff}
	Let $X$ be a locally compact space, $Y$ a set, $\mi$ a positive Radon measure on $X$, and  $\pi$ a  mapping from $X$ into $Y$.
	We say that two points $x,x'$ of $\Supp{\mi}$ are $(\mi,\pi)$-connected if $\pi(x)=\pi(x')$ and there are $x=x_1,\dots, x_k=x'\in \pi^{-1}(\pi(x))\cap \Supp{\mi}$ such that, for every $j=1,\dots, k$, for every neighbourhood $U_j$ of $x_j$ in $\Supp{\mi}$, and for every neighbourhood $U_{j+1}$ of $x_{j+1}$ in $\Supp{\mi}$, the set $ \pi^{-1}(\pi(U_j)\cap \pi(U_{j+1})) $ is not $\mi$-negligible.
	We say that $\mi$ is $\pi$-connected if every pair of elements of $ \Supp{\mi}$ having the same image under $\pi$ are $(\mi,\pi)$-connected.
\end{deff}

\begin{prop}\label{cor:A:6}\label{prop:A:8}
	Let $E_1,E_2$ be two finite-dimensional affine spaces, $L\colon E_1\to E_2$ an affine mapping and $\mi$ a positive Radon measure on $E_1$. Assume that the support of $\mi$ is either a convex set and that $L$ is proper on it, or that the support of $\mi$ is the boundary of a convex polyhedron on which $L$ is proper
	Then, $\mi$ is $L$-connected.
\end{prop}

\begin{proof}
	{\bf 0.} The assertion is~\cite[Proposition 6.3]{Calzi} when the support of $\mi$ is convex. Then, assume that the support of $\mi$ is the boundary of a convex polyhedron $C$, on which $L$ is proper.
	
	{\bf1.} Consider first the case in which $C$ is compact and has non-empty interior, $E_1=\R^{n}$, $E_2=\R^{n-1}$ and $L(x_1,\dots,x_n)=(x_1,\dots, x_{n-1})$ for every $(x_1,\dots, x_n)\in E_n$.
	Define $C'\coloneqq L(C)$, so that $C'$ is a compact convex polyhedron of $E_2$. 
	Now, the functions
	\[
	f_-\colon C'\ni x'\mapsto \min\Set{y\in \R\colon (x',y)\in C}
	\qquad \text{and}\qquad
	f_+\colon C'\ni x'\mapsto \max\Set{y\in \R\colon (x',y)\in C}
	\]
	are well-defined; in addition, $f_-$ is convex while $f_+$ is concave. 	
	Therefore, $f_-$ and $f_+$ are continuous on $\open {C'}$ by~\cite[Corollary to Proposition 21 of Chapter II, § 2, No.\ 10]{BourbakiTVS}.
	Now, observe that $f_-\meg f_+$; if $f_-(x')=f_+(x')$ for some $x'\in \open{C'}$, then $f_-=f_+$  on $C'$ by convexity, and this contradicts the assumption that $C$ has non-empty interior. 
	Therefore, $\Set{(x',y)\colon x'\in \open{C'}, f_-(x')<y<f_+(x')}$ is the interior of $C$, 
	so that $\partial C \cap (\open {C'} \times \R)$ is the union of the graphs $\Gamma_-$ and $\Gamma_+$ of the restrictions of $f_-$ and $f_+$ to $\open {C'}$.
	Since $L$ induces homeomorphisms of $\Gamma_-$ and $\Gamma_+$ onto $\open {C'}$, it follows that $(x',f_-(x'))$ and $(x', f_+(x'))$ are $(\mi,L)$-connected for every $x'\in \open {C'}$.
	
	Now, take $x'\in \partial C'$, and observe that $\Set{(x',y)\colon y\in [f_-(x'), f_+(x')]}\subseteq \partial C$.
	Take $y\in [f_-(x'), f_+(x')]$ and
	an $(n-1)$-dimensional facet $F$ of $C$ which contains $(x',y)$.
	Observe that the support of $\chi_F\cdot \mi$ is $F$. 
	Indeed, clearly $\Supp{\chi_F\cdot \mi}\subseteq F$. 
	Conversely, take $x$ in the relative interior of  $F$. Then, every sufficiently small open neighbourhood of $x$ intersects $\partial C$ only on $F$, so that it is clear that $x\in \Supp{\chi_F\cdot \mi}$. Since $F$ is the closure of its relative interior, the assertion follows.
	Then~{\bf0} implies that $(x',y),(x',y')$ are $(\mi,L)$-connected for every $y'\in \R$ such that $(x',y')\in F$.  
	Since $\partial C$ is the (finite) union of its $(n-1)$-dimensional facets, it follows that $(x,y), (x,y')$ are $(\mi,L)$-connected for every $y'\in  [f_-(x'), f_+(x')]$. 
	The assertion follows in this case.	
	
	{\bf2.} Now, consider the general case. 
	Observe first that we may assume that $C$ has non-empty interior.
	Then, take $y\in L(\partial C)$ and a closed cube $Q$ in $E_2$ which contains $y$ in its interior. 
	Then, $C \cap L^{-1}(Q)$ is a compact polyhedron; in addition, 
	\[
	\partial C\cap  L^{-1}(\open Q)=\partial [C\cap L^{-1}(Q)]\cap L^{-1}(\open Q).
	\]
	Hence, in order to prove that any two points of $L^{-1}(y)\cap \partial C$ are $(\mi,L)$-connected, we may assume that $C$ is compact.
	Now, take $x_1,x_2\in \partial C$ such that $x_1\neq x_2$ and $L(x_1)=L(x_2)$. 
	Let $L'$ be an affine mapping defined on $E_1$ such that $L'(x_1)=L'(x_2)$ and such that the fibres of $L'$ have dimension $1$. 
	Then, we may apply~{\bf1} above and deduce that $x_1,x_2$ are $(\mi,L')$-connected. 
	It is then easily seen that $x_1,x_2$ are also $(\mi,L)$-connected, whence the result.
\end{proof}

\begin{oss}
	Notice that Proposition~\ref{cor:A:6} is false when the support of $\mi$ is the boundary of a more general convex set (on which $L$ is proper).
	Indeed, choose $E_1=\R^3$, $E_2=\R^2$, $L=\pr_{1,2}$ and
	\[
	C_1\coloneqq\Set{(x,y,z)\in E_1\colon  2 y z\Meg x^2 , z\in [0,1], y\Meg 0}.
	\]
	Define $C$ as the union of $C_1$ and $\pi(C_1)$, where $\pi$ is the reflection along the plane $\pr_3^{-1}(1)$. 
	Then, $\partial C$ is the union of
	\[
	C'_1\coloneqq \Set{(x,y,z)\in E_1\colon  2 y z= x^2 , z\in [0,1], y\Meg 0}
	\]
	and $\pi(C'_1)$. 
	Choose any continuous function $m_1\colon C'_1\to \C$, and define $m\colon \partial C\to \C$ so that it equals $m_1$ on $C'_1$ and $m_1\circ \pi$ on $\pi(C'_1)$. 
	Then, $m$ is clearly continuous. In addition, it is clear that $C'_1$ intersects the fibres of $L$ at one point at most, except for $L^{-1}(0,0)$. Since $m$ can be chosen so that it is \emph{not} constant on $\Set{(0,0)}\times [0,2]$, Proposition~\ref{prop:A:7} below shows that $\chi_{\partial C}\cdot \Hc^2$ cannot be $L$-connected. 
\end{oss}

\begin{prop}[\cite{Calzi}, Proposition 6.2]\label{prop:A:7}
	Let $X,Y,Z$ be three locally compact spaces, $\pi\colon X\to Y$ a $\mi$-measurable mapping, and $\mi$ a  $\pi$-connected positive Radon measure on $X$.  
	Assume that $\pi_*(\mi)$ is a Radon measure and that there is a disintegration $(\lambda_y)_{y\in Y}$ of $\mi$ relative to $\pi$ such that $\Supp{\lambda_y}\supseteq\Supp{\mi} \cap \pi^{-1}(y)$ for $\pi_*(\mi)$-almost every $y\in Y$.
	
	Take a continuous mapping $m_0\colon X\to Z$ such that there is mapping $m_1\colon Y\to Z$ such that $m_0(x)= (m_1\circ \pi)(x)$ for $\mi$-almost every $x\in X$. 
	Then, there is a $\pi_*(\mi)$-measurable mapping $m_2\colon Y\to Z$ such that $m_0=m_2\circ \pi$ \emph{pointwise} on $\Supp{\mi}$.
\end{prop}

Concerning the assumption on the disintegration, we shall often make use of a general result by Federer (cf.~\cite[Theorem 3.2.22]{Federer}), which basically provides the disintegration of a wide family of measures. We shall also derive Lemma~\ref{cor:6} from it.

\smallskip

For what concerns the composition of Schwartz functions, the techniques employed to prove~\cite[Theorem 6.1]{AstengoDiBlasioRicci2} can be effectively used to derive from~\cite[Theorem 0.2]{BierstoneMilman} and~\cite[Theorem 0.2.1]{BierstoneSchwarz} the following result:
 
\begin{teo}[\cite{Calzi}, Theorem 7.2]\label{teo:7}
	Let $P\colon \R^n\to \R^m$ be a polynomial mapping, and assume that $\R^n$ and $\R^m$ are endowed with dilations such that $P(r\cdot x)=r\cdot P(x)$ for every $r>0$ and for every $x\in \R^n$.
	Let $C$ be a dilation-invariant subanalytic closed subset of $\R^n$, and assume that $P$ is proper on $C$ and that $P(C)$ is Nash subanalytic.
	Then, the canonical mapping
	\[
	\Phi\colon\Sc(\R^m)\ni \phi\mapsto \phi\circ P\in \Sc_{\R^n}(C)
	\]
	has a closed range and admits a continuous linear section defined on $\Phi(\Sc(\R^m))$.
	In addition, $\psi\in \Sc_{\R^n}(C)$ belongs to the image of $\Phi$ if and only if it is a `formal composite' of $P$, that is,   for every $y\in \R^m$ there is $\phi_y\in \Ec(\R^m)$ such that, for every $x\in C\cap P^{-1}(y)$, the Taylor series of $\phi_y\circ P$ and $\psi$ at $x$ differ by the Taylor series of a function of class $C^\infty$ which vanishes on $C$.
\end{teo}

In the statement, we denoted by $\Sc_{\R^n}(C)$ the quotient of $\Sc(C)$ by the space of $f\in \Sc(\R^n)$ which vanish on the closed set $C$.
We refer the reader to~\cite{BierstoneMilman, BierstoneMilman2, BierstoneSchwarz} for the notion of (Nash) subanalytic sets; as a matter of fact, in the applications we shall only need to know that any convex subanalytic set is automatically Nash subanalytic, since it is contained in an affine space of the same dimension, and that semianalytic sets are Nash subanalytic (cf.~\cite[Proposition 2.3]{BierstoneMilman}).

\begin{cor}[\cite{Calzi}, Corollary 7.3]\label{cor:A:7}
	Let $V$ and $W$ be two finite-dimensional vector spaces,  $C$ a subanalytic closed convex cone in $V$, and $L$ a linear mapping of $V$ into  $W$ which is proper on $C$. 
	Take $m_1\in \Sc(V)$ and assume that there is $m_2\colon W\to \C$ such that $m_1=m_2\circ L$ on $C$. 
	Then, there is $m_3\in \Sc(W)$ such that $m_1=m_3\circ L$ on $C$.
\end{cor}

\subsection{Equivalence and Completeness}

Let us now add some definitions to those presented in~\cite{Calzi}.

\begin{deff}\label{def:1:1}
	We say that two Rockland families $\Lc_{A_1}$ and $\Lc_{A_2}$ are functionally equivalent if there are two Borel functions $m_1\colon E_{\Lc_{A_1}}\to E_{\Lc_{A_2}}$ and $m_2\colon E_{\Lc_{A_2}}\to E_{\Lc_{A_1}}$ such that $m_1(\Lc_{A_1})=\Lc_{A_2}$ and $m_2(\Lc_{A_2})=\Lc_{A_1}$.
	
	We shall say that a Rockland family $\Lc_A$ is functionally complete if every $\beta_{\Lc_A}$-measurable function $m\colon E_{\Lc_A}\to \C$ such that $m(\Lc_A)$ is a differential operator equals a polynomial $\beta_{\Lc_A}$-almost everywhere. 
\end{deff}

Notice that there exist Rockland families which are not functionally complete; for example, if $\Lc$ is a positive Rockland operator, then $(\Lc^2)$ is a Rockland family which is not functionally complete.
Further, observe that we cannot talk of a `completion' of $\Lc_A$ unless we know that the algebra of differential operators arising as functions of $\Lc_A$ is (algebraically) finitely generated.

The main point for considering functional completeness is the following result, which shows that property $(S)$ implies functional completeness; nevertheless, the converse fails in general (cf.~Proposition~\ref{prop:4}).

\begin{prop}\label{prop:12:1}
	Let $\Lc_A$ be a Rockland family on a homogeneous group $G$. If $\Lc_A$ satisfies property $(S)$, then it is functionally complete.
\end{prop}

\begin{proof}
	Take a function of $\Lc_A$ which is a left-invariant differential operator of degree $\delta$, and let $T$ be its convolution kernel;
	assume that $\Lc_A$ satisfies property $(S)$. Take $\tau\in \Sc(E_{\Lc_A})$ such that $\tau(\lambda)\neq 0$ for every $\lambda\in E_{\Lc_A}$; then $\Kc_{\Lc_A}(\tau)*T\in \Sc(G)$, so that it has a multiplier $m_1\in \Sc(E_{\Lc_A})$. 
	If we define $m\coloneqq \frac{m_1}{\tau}$, then $m\in C^\infty(E_{\Lc_A})$ and $\Kc_{\Lc_A}(m)=T$.
	By means of~\cite[Theorem 1.37]{FollandStein}, we see that there are a family with finite support $(P_{\delta'})_{0\meg 0 \meg \delta}$ of homogeneous polynomials, where $P_{\delta'}$ has homogeneous degree $\delta'$ for every $\delta'\in [0,\delta]$, and a function $\omega$, such that
	\[
	m(\lambda)=\sum_{0\meg \delta'\meg \delta} P_{\delta'}(\lambda) +\omega(\lambda)
	\]
	for every $\lambda\in E_{\Lc_A}$, and such that
	\[
	\lim_{\lambda \to 0}\frac{\omega(\lambda)}{\abs{\lambda}^\delta}=0.
	\]
	Now, clearly $m(r \cdot \lambda)=r^\delta m(\lambda)$ for every $r>0$ and for every $\lambda\in \sigma(\Lc_A)$; fix a non-zero $\lambda\in \sigma(\Lc_A)$. Then,
	\[
	r^\delta m(\lambda)=m(r \cdot \lambda)=\sum_{0\meg\delta'\meg \delta}P_{\delta'}(r \cdot\lambda)+ \omega\left(r\cdot \lambda\right)= \sum_{0\meg\delta'\meg \delta} r^{\delta'} P_{\delta'}(\lambda)+ o\left(r^{\delta}\right)
	\]
	for $r\to 0^+$, so that $P_{\delta'}(\lambda)=0$ for every $\delta'\in [0,\delta[$ and $P_\delta(\lambda)=m(\lambda)$. 
	Therefore,  $m=P_\delta$ on $\sigma(\Lc_A)$, so that $m=P_\delta$ $\mi_{\Lc_A}$-almost everywhere. 
	By the arbitrariness of $T$, the assertion follows (cf.~\cite{RobbinSalamon}).
\end{proof}

\section{Abelian Groups}
\label{sec:3}

In this section, $G$  denotes a homogeneous \emph{abelian} group. In other words, $G$ is the euclidean space $\R^n$ endowed with dilations of the form $r\cdot x =(r^{\dd_1}x_1,\dots, r^{\dd_n}x_n)$ for $r>0$, $x=(x_1,\dots, x_n)\in \R^n$, and some fixed $\dd_1,\dots, \dd_n>0$.
Then,  $\partial=(\partial_1,\dots, \partial_n)$ is a homogeneous basis of the Lie algebra of $G$.
We shall consequently put a scalar product and the associated Hausdorff measures on $G$, and identify the Fourier transform $\Fc$ with a mapping from $\Sc'(G)$ onto $\Sc'(E_{-i\partial})$.

\begin{prop}\label{prop:6}
	Let $P$ be a polynomial mapping with homogeneous components from $E_{-i\partial}$ into $\R^\Gamma$ for some finite set $\Gamma$. Then, $\Lc_A=P(-i\partial)$ is a Rockland family if and only if $P$ is proper.
	In this case, the following hold:
	\begin{enumerate}
		\item $\sigma(\Lc_A)=P(E_{-i\partial})$;
		
		\item a $\beta_{\Lc_A}$-measurable function $m$ admits a kernel in the sense of Definition~\ref{def:1} if and only if $m\circ P$ is a polynomial times an element of $L^2(E_{-i\partial})$;
		in this case,
		\[
		\Kc_{\Lc_A}(m)=\Fc^{-1}(m\circ P).
		\]
	\end{enumerate}
\end{prop}

\begin{proof}
	Since $\sigma(-i\partial)=E_{-i\partial}$ and $-i\partial$ is Rockland, the assertions follow easily
	by the properties of the Fourier transform.
\end{proof}

By means of~\cite[Theorem 3.2.22]{Federer}, one may obtain some relatively explicit formulae for $\beta_{\Lc_A}$ and $\chi_{\Lc_A}$.

\smallskip

In the following result, we give complete answers to our main questions in the case of \emph{one} operator.

\begin{teo}\label{teo:9:3}
	Let $\Lc$ be  a positive Rockland operator on  $G$.\footnote{Notice that $\Lc=P(-i\partial)$ where $P$ is a proper polynomial; unless $G=\R$, in which case our analysis is trivial, $P$ must have a constant sign, so that we may assume that $\Lc$ is positive without loss of generality.} Then, $\chi_\Lc$ has a continuous representative which is of class $C^\infty$ on $\R_+^*\times G$; in particular, property $(RL)$ holds.
	
	In addition, take $m\in C_b(\beta_\Lc)$, and let $k$ be the greatest $k'\in \N^*$ such that $P^{\frac{1}{k'}}$ is a polynomial. Then, the following conditions are equivalent:
	\begin{enumerate}
		\item $\Kc_{\Lc}(m)\in \Sc(G)$;
		
		\item there are $m_0,\dots, m_{k-1}\in\Sc(\R)$ such that $m(\lambda)=\sum_{h=0}^{k-1} \lambda^{\frac{h}{k}} m_h(\lambda)$ for every $\lambda\Meg 0$.
	\end{enumerate}
	In particular, property $(S)$ holds if and only if $k=1$.
\end{teo}

Before we pass to the proof of the preceding result, we need to establish a lemma.

\begin{lem}\label{lem:9:2}
	Let $A$ be a non-empty finite set and endow $\R^A$ with a family of (not-necessarily isotropic) dilations. Take a positive, non-constant, homogeneous polynomial $P $ in $\R[A]$ and assume that  there is a homogeneous element $x$ of $\R^A$ such that $P(x)\neq 0$. Then, the following statements are equivalent:
	\begin{enumerate}
		\item there are no positive homogeneous polynomials $Q\in \R[A]$ and no $k\in \N$ such that $k\Meg 2$ and $P=Q^k$;
		
		\item if $m$ is a complex-valued function defined on $\R_+$ such that $m\circ P$ is $C^\infty$ on $\R^A$, then  $m$ may be extended to an element of $C^\infty(\R)$.
	\end{enumerate}
\end{lem}

\begin{proof}
	{\bf1 $\implies$ 2.}  Take $m\colon \R_+\to \C$ and assume that $m\circ P$ is $C^\infty$ on $\R^A$. 
	Notice that there is a homogeneous polynomial $P_x\in \R[X]$ such that $P(\lambda x)=P_x(\lambda)$ for every $\lambda\in \R$.\footnote{Notice that $\lambda x$ denotes the scalar multiplication of $x$ by $ \lambda$, \emph{not} the dilate $\lambda \cdot x$ of $x$ by $\lambda$, which is meaningful only for $\lambda>0$.}  
	In particular, $m\circ P_x$ is of class $C^\infty$. 
	In addition,  $P_x(X)= a_x X^{d_x}$ for some $a_x\neq 0$ and $d_x\in\N^*$; we may assume that $a_x=1$.
	It is then clear that $m$ is $C^\infty$ on $\R_+^*$; further, $m\circ P_{x}$  admits a Taylor series $\sum_{j\in\N} \widetilde a_j X^{ j}$ at $0$. 
	Therefore, $m$ admits the asymptotic development $\sum_{j\in\N} a_{x,j} \lambda^{\frac{j}{d_x}}$ for $\lambda\to 0^+$.  
	Suppose that there are some $j\in\N \setminus \left(d_x \N  \right)$ such that $a_{x,j}\neq 0$, and let $j_{x}$ be the least of them.
	Let $q_{x}, r_{x}$ be the quotient and the remainder, respectively, of the division of $j_{x}$ by $d_{x}$.

	Define $\widetilde m\coloneqq m-\sum_{j=0 }^{j_{x} d_{x}} a_{x,j} (\,\cdot\,)^{\frac{j}{d_{x}}}$. Then, $\widetilde m \circ P_{x}$ is $C^\infty$ and $(\widetilde m\circ P_{x})(\lambda) =o\left( \abs{\lambda}^{j_{x} d_{x} }   \right)$.\footnote{Here, $\abs{\lambda}$ denotes the usual absolute value of $\lambda\in \R$.} Hence, it is not hard to see that $\widetilde m$ may be extended to an element of $C^{j_{x}}(\R)$. Let us then prove that
	\[
	\begin{split}
	\partial_{x}^{j_{x}} P^{\frac{j_{x}}{d_{x}}  } &= \partial_{x}^{j_{x}} (m\circ P) - \sum_{  j=0}^{ q_{x}  }     a_{x,  d_{x} j}  \partial_{x}^{j_{x}} P^j  -\sum_{j=j_{x}+1 }^{ j_{x} d_{x}  }     a_{x,  j}  \partial_{x}^{j_{x}} P^{\frac{j}{d_{x}}} - \partial_{x}^{j_{x}}  (\widetilde m\circ P)
	\end{split}
	\]
	extends to a continuous function on $E\coloneqq \Set{x'\in \R^A\colon P(x')\neq 0}\cup \Set{0}$. Indeed, this is clear for the first two terms, and follows from the above remarks for the fourth one. Let us then consider the third term. Notice that both $\partial_{x}$ and $P$ are homogeneous, and that $\partial_{x}^{j_{x}} P^{\frac{j_{x}}{d_{x}}  }$ is homogeneous of degree $0$ on the  $x$ axis, hence on $\R^A$. 
	Hence, $\partial_{x}^{j_{x}} $ must be homogeneous of degree $d\frac{j_{x}}{d_{x}}  $, where $d$ is the homogeneous degree of $P$. Then, $\partial_{x}^{j_{x}} P^{\frac{j}{d_{x}}} $ is homogeneous of degree $d\frac{j-j_{x}}{d_{x}}>0$, so that it may be extended by continuity at~$0$.

	Therefore, $\partial_{x}^{j_{x}} P^{\frac{j_{x}}{d_{x}}}  $ is a continuous function on $E$ which is homogeneous of degree $0$; hence, it is constant, and its constant value must be $j_x !\neq 0$. Now, Faà di Bruno's formula shows that
	\[
	P^{1-\frac{r_{x}}{d_{x}}}= \frac{1}{j_x !} P^{1-\frac{r_{x}}{d_{x}}}\partial_{x}^{j_{x}} P^{\frac{j_{x}}{d_{x}}}  = \sum_{\sum_{\ell=1}^{j_{x}} \ell \beta_\ell= j_{x}    }  \frac{ 1 }{ \beta!  } \left(\frac{j_{x}}{d_{x}}  \right)_{\abs{\beta}} P^{1+q_{x}- \abs{\beta}} \prod_{\ell=1}^{j_{x}}  \left(  \frac{\partial_{x}^\ell P  }{\ell !} \right)^{\beta_\ell},
	\]
	where $  \left(\frac{j_{x}}{d_{x}}  \right)_{\abs{\beta}} \coloneqq \frac{j_{x}}{d_{x}}\left(\frac{j_{x}}{d_{x}}-1  \right)\dots \left( \frac{j_{x}}{d_{x}}-\abs{\beta}+1  \right)$ is the Pochhammer symbol.
	Then, $P^{1-\frac{r_{x}}{d_{x}}}$ is a rational function, so that there are $N,D\in \R[A]$, with $D\neq 0$, such that $P^{1-\frac{r_{x}}{d_{x}}}=\frac{N}{D}$. 
	Hence, $ P^{d_{x}- r_{x}}=\frac{N^{d_{x}}}{D^{d_{x}}}$, so that $D^{d_{x}}$ divides $N^{d_{x}}$ in $\R[A]$. 
	Since $\R[A]$ is factorial, it follows that $D$ divides $N$, so that $P^{1-\frac{r_{x}}{d_{x}}}$ is a (positive) polynomial.
	Next, let $g$ be the greatest common divisor of $d_{x}$ and $d_{x}-r_{x}$, and take $d',r'\in\N^*$ so that $d_{x}=g \,d'$ and $d_{x}-r_{x}=g\, r'$. Then,
	\[
	\left( P^{1-\frac{r_{x}}{d_{x}}}\right) ^{d'} = P^{r'}.
	\]
	Since $\R[A]$ is factorial, this proves that there is a polynomial $Q\in\R[A]$ such that $Q^{r'}=P^{1-\frac{r_{x}}{d_{x}}}$ and $Q^{d'}=P$.
	Now,  $d'\Meg 2$ since $d_x$ does not divide $d_x-r_x$; in addition, $Q$ is positive since both $P^{1-\frac{r_{x}}{d_{x}}}$ and $P$ are positive and $d',r'$ are coprime: contradiction. Therefore, $a_{x,j}=0$ for every $j\not \in d_{x}\N$, so that the conclusion follows easily.
	
	{\bf 2 $\implies$ 1.} Suppose by contradiction that there are a positive homogeneous polynomial $Q\in \R[A]$ and $k\Meg 2$ such that $P=Q^k$. Define $m\colon \lambda \mapsto \lambda^{\frac{1}{k}}$ on $\R_+$. Then, $m$ is not right-differentiable at $0$; nevertheless, $m\circ P=Q$ since $Q$ is positive, so that $m\circ P$ is $C^\infty$: contradiction.
\end{proof}

\begin{proof}[Proof of Theorem~\ref{teo:9:3}]
	Notice that $\chi_\Lc(\lambda,\,\cdot\,)$ is an eigenfunction of positive type and of class $C^\infty$ of $\Lc$, with eigenvalue $\lambda$, and that $\chi_\Lc(r\cdot \lambda, g)=\chi_\Lc(\lambda, r\cdot g)$ for every $r>0$ and for $(\beta_\Lc\otimes \nu_G)$-almost every $(\lambda,g)$ (cf.~\cite{Tolomeo}). It is then easily seen that $\chi_\Lc$ has a continuous representative which is of class $C^\infty$ on $\R_+^*\times G$.
	
	Now, take $m\in C_b(\beta_\Lc)$ such that $\Kc_{\Lc}(m)\in \Sc(G)$. Then, Proposition~\ref{prop:6} implies that $m\circ P\in \Sc(E_{-i\partial})$.  
	Take a positive polynomial $Q$ on $E_{-i\partial}$ so that $P=Q^k$. 
	Since $[m\circ (\,\cdot\,)^k]\circ Q= m\circ P$, Lemma~\ref{lem:9:2} implies that we may take $\widetilde m\in \Ec(\R)$ so that $m\circ (\,\cdot\,)^k=\widetilde m$ on $\R_+$. In addition, it is clear that we may assume that $\widetilde m\in \Sc(\R)$. 
	Now, let $\sum_{\ell\in \N} a_\ell \lambda^\ell$ be the Taylor development of $ \widetilde m$ at $0$. Take, for $h=1,\dots, k-1$, $m_{h}\in C^\infty_c(\R)$ so that its Taylor development at $0$ is $\sum_{\ell\in\N} a_{h+k\ell} \lambda^\ell$ (cf.~\cite[Theorem 1.2.6]{Hormander2}), and define $m_{0}\coloneqq m-\sum_{h=1}^{k-1} (\,\cdot\,)^{\frac{h}{k}  } m_{h}$ on $\R_+$. Since clearly $m_{0}$ has the asymptotic development $\sum_{\ell\in\N} a_{k\ell} \lambda^\ell$ for $\lambda\to0$ and since $m_{0}\circ (\,\cdot\,)^k$ is of class $C^\infty$, it is easily seen that $m_{0}$ may be extended to an element of $\Sc(\R)$.  Therefore, $m(\lambda)=\sum_{h=0}^{k-1} \lambda^{\frac{h}{k}  } m_{h}(\lambda)$ for every $\lambda\Meg 0$.
	
	Conversely, suppose that there are $m_1,\dots, m_{k-1}\in\Sc(\R)$ such that $m(\lambda)=\sum_{h=0}^{k-1} \lambda^{\frac{h}{k}} m_h(\lambda)$ for every $\lambda\Meg0$. Then, $m\circ P\in \Sc(E_{-i\partial})$, so that $\Kc_{\Lc}(m)\in \Sc(G)$ by Proposition~\ref{prop:6}.
\end{proof}

\begin{cor}
	Let $L\colon \R^n\to \R^{n'}$ be a linear mapping which is proper on $\R_+^n$. Then, $L(-\partial_1^2,\,\dots\,, -\partial_n^2)$ satisfies properties $(RL)$ and $(S)$.
\end{cor}

\begin{proof}
	This is a consequence of Theorems~\ref{teo:4} and~\ref{teo:9:3} when $L$ is the identity. The general case then follows by means of Propositions~\ref{prop:A:8} and~\ref{prop:A:7}, and Corollary~\ref{cor:A:7}.
\end{proof}

\section{$MW^+$ Groups}\label{sec:4}

\begin{deff}\label{def:3:1}
	Let $G$ be a $2$-step stratified group, that is, a simply connected Lie group whose Lie algebra is decomposed as $\gf=\gf_1\oplus \gf_2$ with $\gf_2=[\gf,\gf]$ and $[\gf,\gf_2]=0$. 
	For every $\omega\in \gf_2^*$, define
	\[
	B_\omega\colon \gf_1\times \gf_1\ni(X,Y)\mapsto \langle \omega, [X,Y]\rangle.
	\]
	We say that $G$ is an $MW^+$ group if $B_\omega$ is non-degenerate for some $\omega\neq0$. 
	We say that  $G$ is a Métivier group if it is not abelian and $B_\omega$ is non-degenerate for every $\omega\neq 0$.
	A Heisenberg group is a Métivier group with one-dimensional centre.
\end{deff}

Notice that a $2$-step stratified group satisfies property $MW^+$ if and only if it satisfies the Moore-Wolf condition (cf.~\cite{MooreWolf}) and $[\gf,\gf]$ is the centre of $\gf$.

We shall endow a $2$-step stratified group with the canonical dilations, so that
\[
r\cdot (X+Y)=r X+ r^2 Y
\]
for every $r>0$, for every $X\in \gf_1$ and for every $Y\in \gf_2$. Since $\exp_G\colon \gf\to G$ is a diffeomorphism, these dilations transfer to $G$.

Now, to every symmetric bilinear form $Q$ on $\gf_1^*$ we associate a differential operator on $G$ as follows:
\[
\Lc\coloneqq -\sum_{\ell,\ell'} Q(X_\ell^*, X_{\ell'}^*) X_{\ell} X_{\ell'},
\] 
where $(X_\ell)$ is a basis of $\gf_1$ with dual basis $(X^*_\ell)$. As the reader my verify, $\Lc$ does not depend on the choice of $(X_\ell)$; actually, one may prove that $-\Lc$ is the symmetrization of the quadratic form induced by $Q$ on $\gf^*$ (cf.~\cite[Theorem 4.3]{Helgason}).

By a `sum of squares' we means a differential operator of the form $\Lc=-\sum_{j=1}^k Y_j^2$, where $Y_1,\dots, Y_k$ are elements of $\gf$. If, in addition, $Y_1,\dots, Y_k$ generate $\gf$ as a Lie algebra, then we say that $\Lc$ is a sub-Laplacian. Thanks to~\cite{Hormander}, this is equivalent to saying that $\Lc$ is hypoelliptic.

\begin{lem}
	Let $Q$ be a symmetric bilinear form on $\gf_1^*$, and let $\Lc$ be the associated operator. 
	Then, $\Lc$ is formally self-adjoint if and only if $Q$ is real. 
	In addition, $\Lc$ is formally self-adjoint and hypoelliptic if and only if $Q$ is non-degenerate and either positive or negative.
\end{lem}

\begin{deff}
	Let $V$ be a vector space and $\Phi$ a bilinear form on $V$. Then, we  define
	\[
	\dd_\Phi\colon V\ni v \mapsto \Phi(\,\cdot\,,v)\in V^*.
	\]
\end{deff}

\begin{prop}\label{prop:11}
	Let $Q_1$ and $Q_2$ be two symmetric bilinear forms on $\gf_1^*$, and let $\Lc_1$ and $\Lc_2$ be the associated operators. Then, $\Lc_1$ and $\Lc_2$ commute if and only if
	\[
	\dd_{Q_1}\circ \dd_{B_\omega}\circ \dd_{Q_2}=\dd_{Q_2}\circ \dd_{B_\omega}\circ \dd_{Q_1}
	\]
	for every $\omega\in \gf_2^*$.
\end{prop}

\begin{proof}
	Choose a basis $(X_j)_{j\in J}$ of $\gf_1$ and a basis $(T_k)_{k\in K}$ of $\gf_2$. Let $(X_j^*)_{j\in J}$ and $(T_k^*)_{k\in K}$ be the corresponding dual bases. Define $a_{h,j_1,j_2}\coloneqq Q_h(X_{j_1}^*, X_{j_2}^*)$ for $h=1,2$ and for every $j_1,j_2\in J$, so that $\dd_{Q_h}$ is identified with the matrix $A_h\coloneqq (a_{h,j_1,j_2})_{j_1,j_2\in J}$ for $h=1,2$.
	Analogously, define $b_{k, j_1,j_2}\coloneqq B_{T_k^*}(X_{j_1},X_{j_2})$ for every $k\in K$ and for every $j_1,j_2\in J$, so that $\dd_{B_{T_k^*}}$ is identified with the matrix $B_k\coloneqq(b_{k, j_1,j_2})_{j_1,j_2\in J}$ for every $k\in K$. 
	Now, define $Y_{j_1,j_2}\coloneqq \frac{1}{2}(X_{j_1} X_{j_2}+X_{j_2}X_{j_1})$ for every $j_1,j_2\in J$.
	Then,
	\[
	\Lc_h= \sum_{j_1,j_2\in J} a_{h,j_1,j_2} Y_{j_1,j_2}
	\]
	since $Q_h$ is symmetric. In addition, for every $j_1,j_2,j_3,j_4\in J$,
	\[
	[Y_{j_1,j_2}, Y_{j_3,j_4}]= Y_{j_2,j_4} [X_{j_1},X_{j_3}] +Y_{j_2, j_3}  [X_{j_1}, X_{j_4}] + Y_{j_1,j_4}[X_{j_2}, X_{j_3}] + Y_{j_1,j_3}[X_{j_2}, X_{j_4}] 
	\]
	since the elements of $\gf_2= [\gf_1,\gf_1]$ lie in the centre of $\Uf(\gf)$. Next, observe that, for every $j_1,j_2\in J$,
	\[
	[X_{j_1},X_{j_2}]=\sum_{k\in K} b_{k,j_1,j_2} T_k.
	\]
	Therefore,
	\[
	\begin{split}
	[\Lc_1,\Lc_2]&= \sum_{j_1,j_2,j_3,j_4\in J} \sum_{k\in K} a_{1,j_1,j_2} a_{2,j_3,j_4} [b_{k,j_1,j_3} Y_{j_2,j_4}+ b_{k,j_1,j_4} Y_{j_2, j_3}+ b_{k,j_2,j_3} Y_{j_1,j_4}+ b_{k,j_2,j_4} Y_{j_1,j_3}] T_k\\
	&=2 \sum_{j_1,j_2\in J} \sum_{k\in K} c_{k,j_1,j_2} Y_{j_1,j_2}T_k,
	\end{split}
	\]
	where
	\[
	c_{k,j_1,j_2}=\sum_{j_3,j_4\in J} (a_{1,j_1,j_3}a_{2,j_2,j_4}+a_{1,j_2,j_3}a_{2, j_1,j_4 }) b_{k,j_3,j_4}
	\] 
	for every $k\in K$ and for every $j_1,j_2\in J$. Now, the distinct monomials in the family of the $Y_{j_1,j_2} T_k$, as $j_1,j_2\in J$ and $k\in K$, are linearly independent (cf., for example,~\cite[Theorem 1 of Chapter I, § 2, No.\ 7]{BourbakiLie1}). In addition, denote by $C_k$ the matrix $(c_{k,j_1,j_2})_{j_1,j_2\in J}$ for every $k\in K$. Since $A_{1}$ and $A_{2}$ are symmetric and since $B_k$ is skew-symmetric, we have
	\[
	C_k=A_1 B_k A_2+ A_2 \trasp B_k A_1= A_1 B_k A_2- A_2 B_k A_1
	\]
	for every $k\in K$. The assertion follows easily.
\end{proof}

Now we shall present some results which will enable us to put our homogeneous sub-Laplacians in a particularly convenient form. We state them in terms of the associated quadratic forms.

\begin{prop}\label{prop:10:7}
	Let $(V,\sigma)$ be a finite-dimensional symplectic vector space over $\R$. Let $(Q_\alpha)_{\alpha\in A}$ be a family of positive, non-degenerate bilinear forms on $V$ such that the $\dd_{Q_{\alpha}}^{-1}\circ \dd_\sigma $, as $\alpha$ runs through $A$, commute.
	
	Then, there is a finite family $(P_\gamma)_{\gamma\in \Gamma}$ of projectors of $V$ such that the following hold:
	\begin{itemize}
		\item $P_\gamma$ is $\sigma$- and $Q_\alpha$-self-adjoint for every $\alpha\in A$ and for every $\gamma\in \Gamma$;
		
		\item $I_V=\sum_{\gamma\in \Gamma} P_\gamma$ and $P_\gamma P_{\gamma'}=0$ for $\gamma,\gamma'\in \Gamma$, $\gamma \neq \gamma'$; 
		
		\item the bilinear forms $Q_\alpha(P_\gamma\,\cdot\,,P_\gamma\,\cdot\,)$, as $\alpha\in A$, are all multiples of one another for every $\gamma\in \Gamma$, $\gamma \neq \gamma_0$.
	\end{itemize}
\end{prop}

For the proof, basically follow that of~\cite[Theorem 3.1 (c)]{Hormander} using commutativity in order to get simultaneous diagonalizations. 
Applying~\cite[Theorem 3.1 (c)]{Hormander} (or simply~\cite[Corollary 5.6.3]{AbrahamMarsden}) to the range of each $P_\gamma$, we may find a symplectic basis of $V$ which is $Q_\alpha$-orthogonal for every $\alpha\in A$.

\begin{prop}\label{prop:10:9}
	Take a finite family $(\Lc_\alpha)_{\alpha\in A}$ of commuting homogeneous sub-Laplacians on an $MW^+$ group $G$, and let $(Q_\alpha)_{\alpha\in A}$ be the corresponding family of non-degenerate positive bilinear forms on $\gf_1^*$. 
	Then, there is a finite family $ (P_\gamma)_{\gamma\in \Gamma}$ of non-zero projectors of $\gf_1$ such that the following hold:
	\begin{enumerate}
		\item $I_{\gf_1}=\sum_{\gamma\in \Gamma} P_\gamma$ and $P_{\gamma_1}P_{\gamma_2}=0$ for every $\gamma_1,\gamma_2\in \Gamma$ such that $\gamma_1\neq \gamma_2$;
		
		\item $P_\gamma$ is $B_\omega$- and $\widehat Q_\alpha$-self-adjoint for every $\gamma\in \Gamma$, for every $\omega\in \gf_2^*$, and for every $\alpha\in A$;

		\item for every $\gamma\in \Gamma$, the bilinear forms $ Q_\alpha\left(\trasp P_\gamma\,\cdot\,, \trasp P_\gamma\,\cdot\,\right)$, as $\alpha$ runs through $A$, are mutually  proportional.		
	\end{enumerate}
\end{prop}

\begin{proof}
	Fix $\omega_0\in\gf_2^*$ such that $B_{\omega_0}$ is non-degenerate. 
	Then, Proposition~\ref{prop:10:7} and the remarks which follow its proof imply that there is a basis $X_1,\dots, X_{2 n}$ of $\gf_1$ such that $\dd_{B_{\omega_0}}$ and $\dd_{Q_\alpha}$ are represented by the matrices
	\[
	\left(
	\begin{matrix}
	0 & I\\
	-I & 0
	\end{matrix}
	\right) \qquad \text{and}\qquad
	\left(
	\begin{matrix}
	D_\alpha & 0\\
	0 & D_\alpha
	\end{matrix}
	\right),
	\] 
	respectively,
	for some diagonal matrix $D_\alpha$ ($\alpha\in A$). 
	Denote by $d_{\alpha,1},\dots, d_{\alpha,n}$ the diagonal elements of $D_\alpha$, and denote by $(a_{\omega,j,k})$ the matrix associated with $\dd_{B_\omega}$, for every non-zero $\omega\in \gf_2^*$.

	Now, assume that $A$ has exactly two elements $\alpha_1,\alpha_2$. Then, define
	\[
	\Gamma\coloneqq \left\{ \frac{d_{\alpha_1,j}}{d_{\alpha_2,j}}\colon j\in \Set{1,\dots,n}  \right\}
	\]
	and, for every $\gamma\in \Gamma$, let $V_\gamma$ be the vector subspace of $\gf_1$ generated by the set
	\[
	\left\{ X_j, X_{n+j}\colon  \frac{d_{\alpha_1,j}}{d_{\alpha_2,j}} =\gamma \right\}.
	\]
	Next, take $j,k\in \Set{1,\dots, n}$ such that $\frac{d_{\alpha_1,j}}{d_{\alpha_2,j}}\neq \frac{d_{\alpha_1,k}}{d_{\alpha_2,k}}$. 
	Apply Proposition~\ref{prop:11}, and observe that the  $(j,k)$-th components of (the matrices representing) the equality 
	\[
	\dd_{Q_{\alpha_1}}\circ \dd_{B_\omega} \circ \dd_{Q_{\alpha_2}}=\dd_{Q_{\alpha_2}}\circ \dd_{B_\omega} \circ \dd_{Q_{\alpha_1}}
	\]
	give
	\[
	d_{\alpha_1,j} a_{\omega,j,k} d_{\alpha_2,k}=d_{\alpha_2,j} a_{\omega,j,k} d_{\alpha_1,k},
	\]
	whence $a_{\omega,j,k}=0$. Considering the components $(n+j,k),(j,n+k)$ and $(n+j,n+k)$, we see that $a_{\omega,n+j,k}=a_{\omega,j,n+k}=a_{\omega,n+j,n+k}=0$.
	Therefore, $B_\omega (V_{\gamma_1},V_{\gamma_2})=\Set{0}$ for every non-zero $\omega\in \gf_2^*$ and for every $\gamma_1,\gamma_2\in \Gamma$ such that $\gamma_1\neq \gamma_2$. Then, define $P_\gamma$ as the projector of $\gf_1$ onto $V_\gamma$ with kernel $\bigoplus_{\gamma'\neq \gamma} V_{\gamma'}$.
	The general case follows easily.
\end{proof}

From now on, $G$ will denote an $MW^+$ group,  $(Q_\eta)_{\eta\in H}$  a family of positive symmetric bilinear forms on $\gf_1^*$, and $(T_1,\dots, T_{n_2})$ a basis of $\gf_2$.
We shall denote by $\Lc_\eta $ the sum of squares induced by $Q_\eta $, and we shall assume that $\Lc_A\coloneqq (\Lc_H, (-i T_k)_{k=1,\dots,n_2})$ is a Rockland family.
Observe that this condition is equivalent to the fact that the sum of the $\Lc_\eta$ is hypoelliptic.\footnote{Indeed, if $\pi_0$ is the projection of $G$ onto its abelianization, then $\dd \pi_0(\Lc_A)$ is a Rockland family, so that  $\Fc(\dd \pi_0(\Lc_A))$ vanishes only at $0$. 
Since $\Fc(\dd \pi_0(\Lc_\eta))\Meg 0$ and $\dd \pi_0(T_k)=0$ for every $\eta\in H$ and for every $k=1,\dots, n_2$, this implies that $\sum_{\eta \in H}\Fc(\dd \pi_0(\Lc_\eta))$ vanishes only at $0$, so that $\sum_{\eta\in H} Q_\eta$ is non-degenerate and $\sum_{\eta\in H}\Lc_\eta$ is hypoelliptic.}
We may therefore assume that $Q_\eta$ is non-degenerate for every $\eta\in H$, in which case each $\Lc_\eta$ is a (homogeneous) sub-Laplacian.

We shall also endow $\gf$ with a scalar product for which  $\gf_1$ and $\gf_2$ are orthogonal, 
and which induces $\widehat Q_{\eta _0}$ on $\gf_1$ for some \emph{fixed} $\eta_0\in H$. 
Up to a normalization, \emph{we may then assume that $(\exp_G)_*(\Hc^n)$ is the chosen Haar measure on $G$}, where $n$ is the dimension of $G$.
We  shall endow $\gf_2^*$ with the scalar product induced by that of $\gf_2$, and then with the corresponding Lebesgue measure, that is, $\Hc^{n_2}$. 

Define
\[
J_{Q_\eta, \omega}\coloneqq \dd_{Q_\eta}\circ \dd_{B_\omega}\colon \gf_1 \to \gf_1
\]
for every $\eta\in H$ and for every $\omega\in \gf_2^*$. 

We shall denote by $W$ the set of $\omega\in \gf_2^*$ such that $B_\omega$ is degenerate, so that $G$ is a Métivier group if and only if $W=\Set{0}$.

We shall denote by $\Omega$ the set of $\omega\in \gf_2^*\setminus W$ where $\card\left(\sigma( \abs{J_{Q_H, \omega} } )\right)$ attains its maximum $\overline h$.\footnote{By an abuse of notation, we denote by $\abs{J_{Q_H,\omega}}$ the family $(\abs{J_{Q_\eta,\omega}})_{\eta\in H}$.} 
By means of a straightforward generalization of the arguments of~\cite[§ 1.3--4 and § 5.1 of Chapter II]{Kato}, and taking Proposition~\ref{prop:10:9} into account, one may prove the following results.

\begin{prop}
	The sets $W$ and $\gf_2^*\setminus \Omega$ are algebraic varieties. In addition, there are three analytic mappings
	\begin{align*}
		&\mi\colon \Omega \to ((\R_+^*)^{\overline{h}} )^H\\
		&P\colon \Omega \to \Lc(\gf_1)^{\overline{h}} \\
		&\rho\colon \Omega \to \Set{1,\dots, \overline h}^{n_1}
	\end{align*}
	such that the following hold:
	\begin{itemize}
		\item the mapping
		\[
		\Omega  \ni \omega \mapsto \mi_{\eta,\rho_{k,\omega},\omega}\in \R_+
		\]
		extends to a continuous mapping $\omega \mapsto \widetilde \mi_{\eta,k,\omega}$ on $\gf_2^*$ for every $k=1,\dots, n_1$ and for every $\eta\in H$;
		
		\item for every $h=1,\dots, \overline h$ and for every $\omega\in \Omega$, $P_{h,\omega}$ is a $B_\omega$- and $\widehat Q_H$-self-adjoint projector of $\gf_1$;
		
		\item if $h=1,\dots, \overline h$ and $\omega\in \Omega$, then $\tr P_{h,\omega}= 2\card( \Set{k\in \Set{1,\dots, n_1}\colon \rho_{k,\omega}=h} )$;
		
		\item $\sum_{h=1}^{\overline h} P_{h,\omega}=I_{\gf_1}$ and $\sum_{h=1}^{\overline h} \mi_{\eta,h,\omega} P_{h,\omega}= \abs{J_{Q_\eta, \omega}}$  for every $\omega\in \Omega$ and for every $\eta\in H$.
	\end{itemize}
\end{prop}

\begin{deff}
	Define $\vect{n_1}\colon \Omega \to (\N^*)^{\overline h}$ so that $n_{1,h,\omega}=\frac{1}{2}\tr P_{h,\omega}$ for every $\omega\in \Omega$ and for every $h=1,\dots, \overline h$.
	By an abuse of notation, we shall denote by $(x,t)$ the elements of $G$, where $x\in \gf_1$ and $t\in \gf_2$, thus identifying $(x,t)$ with $\exp_G(x,t)$. 
	For every $x\in \gf_1$, for every $\omega\in \Omega$ and for every $h=1,\dots, \overline{h} $,  define
	\[
	x_{h,\omega}\coloneqq \sqrt{\mi_{\eta _0,h,\omega}}P_{h,\omega} (x).
	\] 
	By an abuse of notation, we shall write $x_\omega$ instead of $(x_{h,\omega})_{h=1,\dots, \overline h}$, and $\abs{x_{\omega}}$ instead of $\left( \sum_{h=1}^{\overline{h} } \abs{x_{h,\omega}}^2\right)^{\sfrac{1}{2}}$.
\end{deff}

The following two results are easy and their proof is omitted.

\begin{cor}\label{cor:11:3}
	The function $\omega \mapsto \mi_{\eta ,\omega}(\vect{n_{1,\omega}})=\widetilde \mi_{\eta ,\omega}(\vect{1}_{n_1})=\frac{1}{2}\norm{ J_{Q , \omega} }_1$ is a norm on $\gf_2^*$ which is analytic on $\gf_2^*\setminus W$ for every $\eta\in H$. 
\end{cor}

\begin{prop}\label{prop:11:3}
	The mapping
	\[
	\gf_1 \times \Omega  \ni (x,\omega)\mapsto \sum_{h=1}^{\overline{h} } x_{h,\omega}=\sqrt[4]{-J_{Q,\omega}^2}(x)
	\]
	extends uniquely to a continuous function on $\gf_1\times\gf_2^*$ which is analytic on $\gf_1 \times (\gf_2^*\setminus W)$.
\end{prop}

\begin{deff}
	Define $G_\omega$, for every $\omega\in \gf_2^*$,  as the quotient of $G$ by its normal subgroup $\exp_G(\ker \omega)$. 
	
	Then, $G_0$ is the abelianization of $G$, and we identify it with $\gf_1$.
	If $\omega\neq 0$, then we shall identify $G_\omega$ with $\gf_1 \oplus \R$, endowed with the product
	\[
	(x_1,t_1) (x_2,t_2)\coloneqq\left(x_1+x_2, t_1+t_2+\frac{1}{2}B_\omega(x_1,x_2) \right)
	\]
	for every $x_1,x_2\in \gf_1$ and for every $t_1,t_2\in \R$. Hence, 
	\[
	\pi_\omega(x,t)=(x,\omega(t))
	\]
	for every $(x,t)\in G$.
\end{deff}

\begin{prop}\label{prop:11:11}
	Define 
	\[
	\widetilde \pi\colon \bigcup_{\omega \in \Omega} \Set{\omega}\times G_\omega\ni (\omega,(x,t))\mapsto \omega\in \Omega ,
	\]
	and identify the domain of $\widetilde \pi$ with $\Omega\times(\gf_1 \oplus \R)$ as an analytic manifold, so that $\varpi$ becomes an analytic submersion. 
	
	Then, $\widetilde \pi$ defines a fibre bundle with base $\Omega$ and fibres isomorphic to $G'\coloneqq \Hd^{n_1}\oplus \R^d$. 
	More precisely, for every $\omega_0\in \Omega$, there is an analytic trivialization $(U,\psi)$ of $\widetilde \pi$ such that the following hold:
	\begin{itemize}
		\item $U$ is an open neighbourhood of $\omega_0$ in $\Omega$;
		
		\item $\psi\colon \widetilde \pi^{-1}(U)\to U\times G'$ is an analytic diffeomorphism such that $\pr_1\circ \psi= \varpi$ and such that $\psi_\omega\coloneqq \pr_2\circ\psi\colon \widetilde \pi^{-1}(\omega)\to G'$ is a group isomorphism for every $\omega\in U$;
		
		\item if $(X_1,\dots,X_{2 n_1}, T, Y_1,\dots, Y_d)$ is a basis of left-invariant vector fields on $G'$ which at the origin induce the partial derivatives along the coordinate axes, then
		\[
		\dd (\psi_\omega \circ\pi_\omega)(\Lc_\eta )= -\sum_{k=1}^d Y_k^2 - \sum_{k=1}^{n_1} \widetilde \mi_{\eta ,k,\omega} (X_{k}^2+ X_{n_1+ k}^2)
		\]
		and
		\[
		\dd (\psi_\omega \circ\pi_\omega)(T_\ell)=\omega(T_\ell) T
		\]
		for every $\eta\in H$, for every $\ell=1,\dots, n_2$, and for every $\omega\in U$.
	\end{itemize}
\end{prop}

The proof is omitted. It basically consists in using the projectors $P_h$ to propagate locally a given basis of eigenvectors and then in `symplectifying' the new basis in order to meet the requirements.

\smallskip

One may then give formulae for $\beta_{\Lc_A}$ and $\chi_{\Lc_A}$ (see~\cite[4.4.1]{Martini} for the general procedure).  
Nevertheless, we shall (almost) only need to know that $\beta_{\Lc_A}$ is equivalent to $\chi_{\Sigma}\cdot \Hc^{n_2}$, where
\[
\Sigma\coloneqq \left\{ (\mi_\omega(\vect{n_{1,\omega}}+2 \gamma),\omega(\vect{T}))\colon \omega\in \gf_2^*, \gamma\in \N^{\overline h}\right\}.
\]

\section{Property $(RL)$}\label{sec:5}

In this section we shall present several sufficient conditions for the validity of property $(RL)$. 
Unlike in the cases considered in~\cite{Calzi}, we are able to prove continuity results for $\chi_{\Lc_A}$, even though under rather strong assumptions (cf.~Theorem~\ref{prop:19:4}); we then deduce property $(RL)$ under slightly weaker assumptions (cf.~Theorem~\ref{prop:19:3}). 
Let us comment a little more on the assumptions of Theorem~\ref{prop:19:3}. Besides the conditions that $\Omega$ is $\gf_2^*\setminus \Set{0}$ and that $\mi$ is constant on the unit sphere associated with the norm $\mi_{\eta_0}(\vect{n_1})$, we need to add the condition that $\dim_\R \mi_\omega(\R^{\overline h})= \dim_\Q \mi_\omega(\Q^{\overline h})$ for some, and then all, $\omega\in \Omega$. 
Even though this condition may appear peculiar, we cannot get rid of it without running into counterexamples, as Theorem~\ref{teo:15:1} shows.
Furthermore, observe that, even though Theorem~\ref{teo:15:1} is the main application of  Theorems~\ref{prop:19:4} and~\ref{prop:19:3}, the latter result can be applied to more general homogeneous sub-Laplacians on $H$-type groups. 
For example, consider the complexified Heisenberg group $\Hd^1_\C$, whose Lie algebra is endowed with an orthonormal basis $X_1,X_2,X_3,X_4,T_1,T_2$ such that $[X_1,X_3]=[X_4,X_2]=T_1$ and $[X_2,X_3]=[X_1,X_4]=T_2$, while the other commutator vanish. 
If $\Lc= -(a X_1^2+ b X_2^2 + c X_3^2+ d X_4^2)$ with $a,b,c,d>0$,  $\sqrt{ \frac{a}{b} }, \sqrt{\frac{c}{d}}\in \Q$, and either $a=b$ or $c=d$, then Theorem~\ref{prop:19:3} applies, but Theorem~\ref{teo:15:1} does not unless $a=b$ and $c=d$.

The next results concern families of the form $(\Lc, (-i T_1,\dots, - i T_{n_2'}))$ for $n_2'<n_2$. Notice that, in this case, we do not only reduce the number of elements of $\gf_2$, but we restrict to the case in which $\card(H)=1$. 
In this case, indeed, the spectrum of $(\Lc, (-i T_1,\dots, - i T_{n_2'}))$ is no longer a countable union of semianalytic sets, but a convex cone, so that things are somewhat easier and we can prove more general results than for the `full family' $\Lc_A$. 
In Theorem~\ref{prop:8}, we show that property $(RL)$ holds if $W=\Set{0}$. 
With reference to the above example in the complexified Heisenberg group, this is the case when $a c \neq b d$ and $a d \neq b c$.

Our last result concerns the case of general $MW^+$ groups (cf.~Theorem~\ref{prop:19:2}); even though its hypotheses are more restrictive than in the preceding one, it nonetheless applies when $G$ is a product of Heisenberg groups and $\Lc$ is a sum of homogeneous sub-Laplacians on each factor (cf.~Proposition~\ref{prop:13}).

\subsection{The Case $n_2'=n_2$}

We begin with a technical lemma.
Here, if $V$ is a finite-dimensional vector space, then $\Ec^0(V)$ denotes the space of continuous functions on $V$ with the topology of locally uniform convergence, while $\Ec'^0_c(V)$ denotes its dual (that is, the space of Radon measures with compact support), endowed with the topology of compact convergence.

\begin{lem}\label{lem:A:23}
	Let $V$ and $\widetilde V$ be two finite-dimensional vector spaces over $\R$, $L$ a discrete subgroup of $V$,  $C$ the convex envelope of $\R_+ F$ for some finite subset $F$ of $L$ which generates $V$, and $\mi\colon V\to \widetilde V$ a linear mapping which is proper on $C$. 
	Assume that $L\cap \ker \mi$ generates $\ker \mi$, and take $\xi\in \mi(C)$.
	Define
	\begin{align*}
		& V_\xi\coloneqq \mi^{-1}(\xi) & &S_\xi\coloneqq  V_\xi\cap C \\
		& n_\xi\coloneqq \dim_\R S_\xi  &
		&\nu_\xi\coloneqq \frac{1}{\Hc^{n_\xi}(S_\xi)}\chi_{S_\xi} \cdot \Hc^{n_\xi}. 
	\end{align*}	
	Take $x_0\in C$ and define, for every $\lambda\in \R^*_+$ and for every $\gamma\in \mi (x_0+ L\cap C) $, 
	\[
	\nu_{\lambda,\gamma}=\frac{1}{c_\gamma}\sum_{\substack{\gamma'\in L\cap C\\\gamma=\mi (x_0+ \gamma')}} \delta_{\lambda(x_0+ \gamma')},
	\]
	where $c_\gamma=\card\left(  \mi^{-1}(\gamma)\cap  (x_0+ L\cap C) \right) $.
	Then,
	\[
	\lim_{ \substack{ (\lambda\gamma,\lambda)\to (\xi,0)\\ \gamma\in  \mi(x_0+  L\cap C)} } \nu_{\lambda, \gamma}=\nu_\xi
	\]
	in $\Ec'^0_c(V)$.
\end{lem}

\begin{proof}
	{\bf1.}  Define $\Sigma\coloneqq \mi(x_0+ L\cap C)$, and define $\Ff_\xi$ as the filter `$(\lambda,\gamma)\in \R_+^*\times \Sigma,\, (\lambda \gamma,\lambda)\to (\xi,0)$.'
	Observe that it will suffice to prove that $\nu_{\lambda, \gamma}$ converges vaguely to $\nu_\xi$ along $\Ff_\xi$. 
	Indeed, the $\nu_{\lambda,\gamma}$ are probability measures supported in
	\begin{equation}\label{eq:A:1}
	S_{\lambda \gamma}\subseteq C\cap \mi^{-1}(K)
	\end{equation}
	eventually along $\Ff_\xi$, where $K$ is a compact neighbourhood of $\xi$ in $\widetilde V$. Since $\mi$ is proper on $C$, the assertion follows.
	
	Now, let us prove that we may reduce to the case in which $x_0=0$. 
	Indeed, define
	\[
	\nu^0_{\lambda,\gamma}\coloneqq \frac{1}{c_\gamma} \sum_{\substack{\gamma'\in L\cap C\\ \gamma=\mi(x_0+\gamma')}} \delta_{\lambda \gamma'}.
	\]
	It will then suffice to prove that $\nu_{\lambda,\gamma}-\nu_{\lambda,\gamma}^0$ converges vaguely to $0$ along $\Ff_\xi$.
	However, take $\phi \in C_c(V)$ and $\eps>0$. 
	Then, there is a neighbourhood $U$ of $0$ in $V$ such that $\abs{\phi(x_1)-\phi(x_2)}<\eps$ for every $x_1,x_2\in V$ such that $x_1-x_2\in U$.
	Therefore, $\abs*{\langle \nu_{\lambda,\gamma}-\nu^0_{\lambda,\gamma},\phi\rangle}<\eps$ as long as $\lambda x_0\in U$, hence eventually along $\Ff_\xi$. The assertion follows.

	{\bf2.} Observe that $C$ is a polyhedral convex cone. In addition, let $n$ be the dimension of $V$, and let $(F_h)_{h\in H}$ be  the (finite) family of $(n-1)$-dimensional facets of $C$; observe that $F_h$ is a convex cone for every $h\in H$, so that $0\in F_h$.
	Take, for every $h\in H$, some $p_h\in V^*$ such that $F_h= \ker p_h\cap C$ and $p_h(C)\subseteq \R_+$. 
	Then, $C$ is the set of $x\in V$ such that $p_h(x) \Meg 0$ for every $h\in H$, and $L\cap \ker p_h$ generates $\ker p_h$ for every $h\in H$.
	
	In addition, let $H_\xi$ be the set of $h\in H$ such that $p_h(S_\xi)=\Set{0}$, and let $H_\xi'$ be its complement in $H$. 
	We shall write $p_{H_\xi}$ and $ p_{H'_\xi}$ instead of $(p_h)_{h\in H_\xi}$ and $(p_h)_{h\in H'_\xi}$, respectively.
	Define $V'_\xi\coloneqq V_\xi\cap \ker p_{H_\xi} $.
	Then, $V'_\xi\cap p_{H'_\xi}^{-1}\left((\R_+^*)^{H'_\xi}\right) $ is the interior of $S_\xi$ in $V'_\xi$; since by convexity $V'_\xi\cap p_{H'_\xi}^{-1}\left((\R_+^*)^{H'_\xi}\right)$ is not empty, 
	 $V'_\xi$ is the affine space generated by $S_\xi$.

	{\bf3.} Define $W_\xi\coloneqq V'_\xi-V'_\xi$, and observe that $L\cap W_\xi$ generates $W_\xi$. 
	Indeed,  the linear mapping $(\mi, p_{H_\xi})\colon V\to \widetilde V\times \R^{H_\xi}$ maps 
	$L$ into the discrete subgroup $ \mi(L)\times \prod_{h\in H_\xi}p_h(L) $ of $\widetilde V\times \R^{H_\xi}$, and $W_\xi$ is the kernel of $(\mi, p_{H_\xi})$, whence the assertion.
	
	Therefore, there are two subspaces $W'_\xi$ and $W''_\xi$ of $V$ such that
	the following hold (cf.~\cite[Exercises 2 and 3 of Chapter VII, § 1]{BourbakiGT2}):
	\begin{itemize}
		\item $W_\xi \oplus W'_\xi=V_0$ and $V_0\oplus W''_\xi=V$;
		
		\item $L\cap W'_\xi$ and $L\cap W''_\xi$ generate $W'_\xi$ and $W''_\xi$, respectively, over $\R$;
		
		\item $(L\cap W_\xi)\oplus (L \cap W'_\xi)\oplus (L\cap W''_\xi)=L$ as abelian groups.
	\end{itemize}
	
	Therefore, we may endow  $V$ and $\widetilde V$ with two scalar products such that  $W_\xi$, $W'_\xi$, and $W''_\xi$ are orthogonal, and $\mi$ induces an isometry of $W''_\xi$ into $\widetilde V$.
	We may further assume that $\norm{p_h}\meg 1$ for every $h\in H$.

	{\bf4.} Define, for $\lambda> 0$ and $\gamma\in \Sigma$,
	\[
	r_{\xi,\lambda,\gamma}\coloneqq \inf\Set{ r>0\colon S_{{\lambda}\gamma}\subseteq B(S_\xi,r)  }+{\lambda},  
	\]
	so that $S_{{\lambda}\gamma  }\subseteq {B}(S_\xi, r_{\xi, \lambda,\gamma})$.
	Let us prove that $r_{\xi,\lambda,\gamma}$ converges to $0$ along $\Ff_\xi$.
	
	Indeed, let $\Uf$ be an ultrafilter finer than $\Ff_\xi$. Denote by $\Kc$ the space of non-empty compact subsets of $V$ endowed with the Hausdorff distance $d_H$. By~\eqref{eq:A:1},~\cite[Proposition 10 of Chapter I, § 6, No.\ 7]{BourbakiGT1} and~\cite[Theorem 6.1]{AmbrosioFuscoPallara}, it follows that $S_{{\lambda}\gamma}$ has a (unique) limit $S$ in $\Kc$ along $\Uf$.
	Now, for every closed neighbourhood $K$ of $\xi$ in $\widetilde V$,
	\[
	S_{{\lambda}\gamma} \subseteq C\cap \mi^{-1}\left(K\right)
	\]
	as long as ${\lambda}\gamma  \in K $, 
	so that, by passing to the limit along $\Uf$, 
	\[
	S\subseteq C \cap  \mi^{-1}\left(K\right).
	\]
	By the arbitrariness of $K$, it follows that $S\subseteq S_\xi$. 
	Therefore,
	\[
	r_{\xi,\lambda,\gamma}\meg  d_H(S, S_{{\lambda}\gamma})+ {\lambda},
	\]
	so that $r_{\xi,\lambda,\gamma}$ tends to $0$ along $\Uf$. 
	Thanks to~\cite[Proposition 2 of Chapter I, § 7, No.\ 2]{BourbakiGT1}, the arbitrariness of $\Uf$ implies that $r_{\xi, \lambda,\gamma}$ converges to $0$ along $\Ff_\xi$.
	
	{\bf5.} Now, let $\pi_\xi$ be the affine projection of $V$ onto $V'_\xi$ with fibres parallel to $W'_\xi \oplus W''_\xi$. 
	Reasoning as in~{\bf1} and taking~{\bf4} into account, we see that  $\nu_{\lambda,\gamma}-(\pi_\xi)_*(\nu_{\lambda,\gamma})$ converges vaguely to $0$ along $\Ff_\xi$, so that it will suffice to prove that $(\pi_\xi)_*(\nu_{\lambda,\gamma})$ converges vaguely to $\nu_\xi$ along $\Ff_\xi$.
	Now, if $n_\xi=0$, then $(\pi_\xi)_*(\nu_{\lambda, \gamma})=\delta_{\xi'}=\nu_\xi$, where $\xi'$ is the unique element of $S_\xi$. Therefore, we may assume that $n_\xi>0$.
	
	Next, take $\eps>0$ and $x,y\in S_{\xi,\lambda, \gamma}\coloneqq \Supp{(\pi_\xi)_*(\nu_{\lambda,\gamma})}$. 
	Assume that $B(x, \eps)\cap p_{H_\xi}^{-1}\left(\R_+^{H_\xi}\right) \subseteq C$, and that $r_{\xi,\lambda,\gamma}<\eps$. 
	Take $y'\in\Supp{\nu_{\lambda, \gamma}}$ such that $\pi_\xi(y')=y$, and let us prove that $y'+x-y\in \Supp{\nu_{\lambda,\gamma}}$.
	Indeed, it is clear that $x-y\in \lambda L\cap W_\xi$, so that $y'+x-y\in \lambda L$. 
	Hence, it will suffice to prove that $y'+x-y\in C$.
	Now, since $y'\in S_{\lambda \gamma}\subseteq B(S_\xi, \eps)$, it follows that there is $x'\in S_\xi$ such that $\abs{y'-x'}<\eps$, so that
	\[
	\eps^2 > \abs{y'-x'}^2= \abs{y-x'}^2+ \abs{y'-y}^2
	\]
	since $y-x'\in W_\xi$ and $y'-y\in W'_\xi\oplus W''_\xi$. 
	Therefore, $\abs{y'-y}<\eps$; since, in addition, $p_h(y'+x-y)=p_h(y')\Meg 0$ for every $h\in H_{\xi}$, it follows that $y'+x-y\in B(x,\eps)\cap p_{H_\xi}^{-1}\left(\R_+^{H_\xi}\right)\subseteq C$. 
	
	{\bf6.} By the arguments of~{\bf5} above, we see that there is a function $c_{\xi, \lambda ,\gamma}$ on $S_{\xi,\lambda, \gamma}$ such that
	\[
	(\pi_\xi)_*(\nu_{\lambda, \gamma})= \sum_{x\in S_{\xi,\lambda, \gamma}} c_{\xi,\lambda, \gamma}(x) \delta_x,
	\]
	and such that $c_{\xi,\lambda, \gamma}(x) \Meg c_{\xi,\lambda, \gamma}(y) $ eventually along $\Ff_\xi$ whenever $x,y\in S_{\xi,\lambda, \gamma}$ and $B(x,\eps)\cap p_{H_\xi}^{-1}\left(\R_+^{H_\xi}\right) \subseteq C$ for some fixed $\eps>0$.
	In particular, $c_{\xi,\lambda, \gamma}$ is constant on the set of $x\in S_{\xi,\lambda, \gamma}$ such that $B(x,\eps)\cap p_{H_\xi}^{-1}\left(\R_+^{H_\xi}\right)\subseteq C$.
	
	Now, let us prove that, if $\eps\meg\min_{h\in H_\xi'} \min_{S_\xi} p_h$ and if $x\in V'_\xi$ and $B(x,\eps)\cap V'_\xi\subseteq C$, then $B(x,\eps)\cap p_{H_\xi}^{-1}\left(\R_+^{H_\xi}\right)\subseteq C$. 
	Indeed, take $x'\in B(x,\eps)$, and assume that $p_h(y)\Meg 0$ for every $h\in H_\xi$. 
	Take $h\in H_\xi'$, and observe that $\abs{p_h(y-x)}\meg \abs{y-x}<\eps$, so that $p_h(y)=p_h(x)+p_h(y-x)\Meg p_h(x)-\eps\Meg0$ by our choice of $\eps$. By the arbitrariness of $h$, it follows that $y\in C$.
	
	{\bf7.} Finally, take a fundamental parallelotope $P_\xi$ of $L\cap W_\xi$, and extend $c_{\xi,\lambda, \gamma}$ to a function on $V$ which is constant on $x+\lambda P_\xi$ for every $x\in \pi_\xi(\lambda L)$, and vanishes outside $S_{\xi,\lambda,\gamma}+\lambda P_\xi$. 
	Then, $\nu_{\xi,\lambda, \gamma}\coloneqq \frac{1}{\Hc^{n_\xi}(\lambda P_\xi)}c_{\xi,\lambda, \gamma}\cdot \Hc^{n'_\xi}$ is a probability measure; in addition, as in~{\bf1} we see that $(\pi_\xi)_*(\nu_{\lambda, \gamma})-\nu_{\xi,\lambda, \gamma}$ converges vaguely to $0$ along $\Ff_\xi$, so that it will suffice to show that $\nu_{\xi,\lambda, \gamma}$ converges vaguely to $\nu_\xi$ along $\Ff_\xi$. 
	However, if $S'_\xi$ denotes the boundary of $S_\xi$ in $V'_\xi$, then~{\bf4} and~{\bf6} imply that $\frac{1}{\Hc^{n_\xi}(\lambda P_\xi)}c_{\xi,\lambda, \gamma}$ is uniformly bounded eventually along $\Ff_\xi$, and converges on $V \setminus S'_\xi$ to a function $g$ which is $0$ on the complement of $S_\xi$, and is constant on $S_\xi \setminus S'_\xi$. 
	The assertion follows by dominated convergence.	
\end{proof}

\begin{teo}\label{prop:19:4}
	Assume that $\Omega=\gf_2^*\setminus \Set{0}$, that  $\dim_\Q \mi_\omega (\Q^{\overline h})=\dim_\R \mi_\omega (\R^{\overline h})$ for every $\omega\in \Omega$, and that $P_h$ is constant on $\Omega$. 
	Then, $\chi_{\Lc_A}$ has a continuous representative.
\end{teo}

\begin{proof}
	{\bf1.} We shall simply write $P_h$ and $\vect{n_1}$ to denote the constant values of the functions $\omega \mapsto P_{h,\omega}$ and $\omega \mapsto \vect{n_{1,\omega}}$, respectively.
	In addition, we shall denote by $\abs{\,\cdot\,}'$ the (homogeneous) norm $ \mi_{\eta _0}(\vect{n_1})$, and by $S'$ the corresponding unit sphere. 
	Choose $\omega_0\in S'$ and observe that $\mi_{\eta ,h,\omega}= \abs{\omega}' \mi_{\eta ,h,\omega_0}$ and $x_{h,\omega}=\sqrt{\abs{\omega}'} x_{h,\omega_0}$ for every $\omega\in \Omega$.
	
	For every $\xi\in \mi_{\omega_0}(\R_+^{\overline h})$, let $\Ff_\xi$ denote the filter `$(\lambda,\gamma)\in  \R^*_+\times \Sigma,\, (\lambda\gamma,\lambda)\to (\xi,0)$,' where $\Sigma\coloneqq \mi(\vect{n_1}+2 \N^H)$.
	In addition, define, for every $\lambda\in \R^*_+$ and for every $\gamma\in \Sigma$, 
	\[
	\nu'_{\lambda,\gamma}=\sum_{\gamma=\mi_{\omega_0}(\vect{n_1}+2 \gamma')} \delta_{\lambda(\vect{n_1}+2 \gamma')},
	\]
	and $\nu_{\lambda,\gamma}\coloneqq \frac{1}{\nu'_{\lambda,\gamma}(\R^{\overline h})}\cdot \nu'_{\lambda,\gamma}$,
	so that $\nu_{\lambda,\gamma}$ is a probability measure.  
	Then, Lemma~\ref{lem:A:23} implies that $\nu_{\lambda,\gamma}$ converges to some probability measure $\nu_\xi$ in $\Ec'^0_c(\R^{\overline h})$ along $\Ff_\xi$.

	{\bf2.} Denote by $\Lambda_\gamma^m$ the $\gamma$-th Laguerre polynomial of order $m$, and by $J_m$ the Bessel function (of the first kind) of order $m$.
	Define, for every $(x,t)\in \R^H\times \R$,
	\[
	\chi_0(\lambda (\vect{n_1}+2\gamma'),\lambda, x,t) = e^{-\frac{1}{4}\lambda\abs{x}^2+i \lambda t} \frac{1}{ \binom{\vect{n_1}+\gamma'-\vect{1}_{\overline h}}{\gamma'} } \prod_{h=1}^{\overline{h}} \Lambda^{n_{1,h}-1}_{\gamma'_h}\left( \frac{1}{2}\lambda\abs{x_h}^2      \right)
	\] 
	for every $\lambda\in \R^*_+$ and for every $\gamma'\in \N^{\overline h}$, and
	\[
	\chi_0(\xi',0, x,t)\coloneqq    \prod_{h=1}^{\overline{h}}  \frac{2^{n_{1,h}-1}(n_{1,h}-1)!}{ \left(\sqrt{ \xi'_h} \abs{x_h}\right)^{n_{1,h}-1 } } J_{n_{1,h}-1}\left(\sqrt{\xi'_h} \abs{x_h}  \right)
	\]
	for every $\xi' \in \R_+^{\overline h}$. 
	Then, $\chi_0$ extends to a continuous function on $\R^{\overline h}\times \R \times \R^{\overline h}\times \R$. 
	
	Next define, for every $\lambda\in \R_+$, 
	\[
	\begin{split}
	f_\lambda\colon \R^{\overline h}\ni x \mapsto& (2\lambda)^{\abs{\vect{n_1}-\vect{1}_{\overline h}}}\binom{ \frac{x}{2 \lambda}+\frac{\vect{n_1}}{2}- \vect{1}_{\overline h}   }{\vect{n_1}-\vect{1}_{\overline h}}=\prod_{h=1}^{\overline{h}} \frac{(x_h+\lambda n_{1,h}-2 \lambda) \cdots (x_h-\lambda n_{1,h}+2 \lambda) }{(n_{1,h}-1)!},
	\end{split}
	\]
	and observe that  $f_\lambda$ converges locally uniformly to $f_0$ as $\lambda \to 0^+$. 
	Therefore, $f_\lambda \cdot \nu_{\lambda, \gamma}$ converges to $f_0\cdot \nu_\xi$ in $\Ec'^0_c(\R^{\overline h})$ along $\Ff_\xi$. 
	If we define $\nu'_{\lambda, \gamma}\coloneqq \frac{1}{\nu_{\lambda,\gamma}(f_\lambda)}f_\lambda \cdot \nu_{\lambda,\gamma}$ and $\nu'_{\xi}\coloneqq \frac{1}{\nu_{\xi}(f_0)}f_0 \cdot \nu_{\xi}$, then $\nu'_{\lambda,\gamma}$ converges to $\nu'_\xi$ in $\Ec'^0_c(\R^{\overline h})$ along $\Ff_\xi$. 
	
	Define, for every  $\omega\in \Omega$ and for every $\gamma\in  \Sigma$,
	\[
	\chi_1((\abs{\omega}'\gamma,\omega(\vect{T})), (x,t))\coloneqq \langle\nu'_{\abs{\omega}',\gamma}, \chi_0(\,\cdot\,,\abs{\omega}',(\abs{x_{h,\omega_0}})_{h=1}^{\overline{h}}, {\textstyle\frac{\omega(t)}{\abs{\omega}'}})  \rangle,
	\]
	so that $\chi_1$ is a representative of $\chi_{\Lc_A}$ (reason as in~\cite[4.4.1]{Martini}, and take~\cite{AstengoDiBlasioRicci} into account). 
	Now,
	\[
	\lim_{(\lambda,\gamma) ,\Ff_\xi } \chi_1((\lambda\gamma,\lambda \omega(\vect{T})), (x,t))=\langle\nu'_\xi, \chi_0(\,\cdot\,,0, (\abs{x_{h,\omega_0}})_{h=1}^{\overline{h}}, \omega(t))\rangle
	\]
	uniformly as $\xi$ runs through $\mi(\R_+^{\overline h}) $, as $\omega$ runs through $S'$, and as $(x,t)$ runs through a compact subset of $G$. 
	Since $ \langle\nu'_\xi, \chi_0(\,\cdot\,,0, (\abs{x_{h,\omega_0}})_{h=1}^{\overline{h}}, \omega(t))\rangle $ does \emph{not} depend on $\omega$, it follows that $\chi_1$ is continuous on $\sigma(\Lc_A)\times G$. 
	The assertion follows from~\cite[Corollary to Theorem 2 of Chapter IX, § 4, No.\ 2]{BourbakiGT2}.	
\end{proof}

\begin{teo}\label{prop:19:3}
	Assume that  $\Omega =\gf_2^*\setminus \Set{0}$, that $\dim_\Q \mi_\omega (\Q^{\overline h})=\dim_\R \mi_\omega ( \R^{\overline h})$ for every $\omega\in \Omega$, and that $\mi$ is constant where $\mi_{\eta_0}(\vect{n_1})$ is constant. 
	Then, $\Lc_A$ satisfies property $(RL)$.
\end{teo}

\begin{proof}
	Take $\phi \in L^1_{\Lc_A}(G)$. 
	Let $S'$ be the unit sphere associated with the homogeneous norm $\abs{\,\cdot\,}'\colon\omega \mapsto  \mi_{\eta _0,\omega} (\vect{n_1})$.
	Now, observe that, arguing as in~\cite[Proposition 5.4]{MartiniRicciTolomeo}, by means of the group Plancherel formula one may prove that, given any $m\in L^\infty(\beta_{\Lc_A})$ such that $\Kc_{\Lc_A}(m)\in L^1(G)$, we have $\pi^*(\Kc_{\Lc_A}(m))=m(\dd \pi(\Lc_A)) $ for \emph{almost every} (class of irreducible unitary representations) $\pi$ in the dual of $G$. 
	Then, comparing the irreducible representations of $G$ and the $G_\omega$, we see that there is a negligible subset $N$ of $S'$ such that $(\pi_\omega)_*(\phi)\in L^1_{\dd \pi_\omega(\Lc_A)}(G_\omega)$ for every $\omega\in S'\setminus N$.
	Observe, in addition, that the mapping
	\[
	\omega \mapsto (\pi_\omega)_*(\phi)\in L^1(\gf_1\oplus \R)
	\]
	is continuous on $\Omega$, hence on $S'$.
	Now, fix $\omega_0\in S'$, and take $(U,\psi)$ as in Proposition~\ref{prop:11:11}. 
	Then, it is easily seen that the mapping
	\[
	U\cap S'\ni \omega \mapsto (\psi_{\omega}\circ \pi_\omega)_*(\phi)\in L^1(G')
	\]
	is continuous. 
	Furthermore, observe that our assumptions imply that, with the notation of Proposition~\ref{prop:11:11},
	\[
	\Lc'_H\coloneqq \dd (\psi_\omega\circ \pi_\omega)(\Lc_H)=\left( - \sum_{k=1}^{d} Y_k^2-\sum_{k=1}^{n_1} \widetilde \mi_{\eta,k}(X_k^2+X_{n_1+k}^2)\right)_{\eta\in H}
	\]
	does \emph{not} depend on $\omega$, while
	\[
	\dd (\psi_\omega\circ \pi_\omega)(\vect{T})= \omega(\vect{T}) T.
	\]
	Observe that $(\Lc'_H,- i T)$ satisfies property $(RL)$ by Theorem~\ref{prop:19:4}, and that $(\psi_\omega \circ \pi_\omega)_*(\phi)\in L^1_{(\Lc'_H, - i T)}(G')$ for every $\omega\in S'\setminus N$, hence for every $\omega\in S'$ by continuity.\footnote{It is easily proved that $L^1_{(\Lc'_H, - i T)}(G')$ is closed in $L^1(G')$.}
	Therefore, the mapping
	\[
	U\cap S'\ni\omega \mapsto \Mc_{(\Lc'_H,- i T)}((\psi_{\omega}\circ \pi_\omega)_*(\phi))\in C_0(\sigma(\Lc'_H,- i T))
	\] 
	is continuous. 
	Now, take $\omega\in U\cap S'\setminus N$. Then,~\cite[Proposition 3.2.4]{Martini}, applied to the right quasi-regular representation of $G'$ in $L^2(G_0)$,  implies that
	\[
	\Mc_{(\Lc'_H,- i T)}((\psi_{\omega}\circ \pi_\omega)_*(\phi))(\lambda,0)=\Mc_{\dd \pi_0(\Lc_H)}((\pi_0)_*(\phi))(\lambda)
	\]
	for every $\lambda\in \R^H$ such that $(\lambda,0)\in \sigma(\Lc'_H,- i T)$, that is, for every $\lambda\in \sigma(\dd \pi_0(\Lc_H))$.
	By continuity, this proves that the mapping $U\cap S'\ni\omega\mapsto \Mc_{(\Lc'_H,- i T)}((\psi_{\omega}\circ \pi_\omega)_*(\phi))(\lambda,0)$ is constant for every $\lambda\in \sigma(\dd \pi_0(\Lc_H))$.
	Taking into account the arbitrariness of $U$, we infer that there is a unique $m\in C_0(\sigma(\Lc_A))$ such that
	\[
	m(\lambda, \omega(\vect{T}))=\Mc_{(\Lc'_H,- i T)}((\psi_{U,\sfrac{\omega}{\abs{\omega}'}} \circ \pi_{\sfrac{\omega}{\abs{\omega}'}})_*(\phi))(\lambda,\abs{\omega}')
	\]
	for every $(\lambda, \omega(\vect{T}))\in \sigma(\Lc_A)$ such that $\frac{\omega}{\abs{\omega}'}\in U\cap S'$, where $U$ runs through a finite covering of $S'$ and $\psi_U$ is the associated local trivialization as above. 
	Hence, $\phi=\Kc_{\Lc_A}(m)$ and the assertion follows.
\end{proof}

Here we prove a negative result.

\begin{prop}\label{prop:19.4}
	Assume that $G$ is the product of $k\Meg 2$ $MW^+$ groups $G_1,\dots, G_k$, and assume that each $G_j$ is endowed with a homogeneous sub-Laplacian $\Lc'_j$. Assume that $\card(H)=1$ and that $\Lc=\Lc'_1+\dots+\Lc'_k$.
	Then, $\Lc_A$ does not satisfy properties $(RL)$ and $(S)$.
\end{prop}

\begin{proof}
	Take, for every $j=1,\dots,k$, a basis $\vect{T}_j$ of the centre $\gf_{j,2}$ of the Lie algebra of $G_j$. 
	Then, we may assume that $\Lc_A=(\Lc, -i \vect{T}_1,\dots, - i \vect{T}_k)$; define  $\Lc'_{A'}\coloneqq (\Lc'_1,\dots, \Lc'_k, -i \vect{T}_1,\dots, - i \vect{T}_k)$. 
	Then, there is a unique linear mapping $L\colon E_{\Lc'_{A'}}\to E_{\Lc_A}$ such that $\Lc_A=L(\Lc'_{A'})$. 
	Now, take $j\in \Set{1,\dots,k}$ and $\gamma\in \N^{\overline h_j}$, and define
	\[
	C_{j,\gamma}\coloneqq \Set{(\mi_{j,\omega}(\vect{n}_{\vect{1},j,\omega}+2 \gamma) ,\omega(\vect{T}_j) )\colon \omega\in \Omega_j  }.
	\]
	Define 
	\[
	C\coloneqq \bigcup_{\gamma\in \prod_{j=1}^k \N^{\overline h_j}} \prod_{j=1}^k C_{j,\gamma_j},
	\]
	and observe that $\beta_{\Lc'_{A'}}$ is equivalent to $\chi_C\cdot \Hc^{n_2}$.
	Next, define
	\[
	N\coloneqq \R^k \times \bigcup_{\substack{\gamma\in\prod_{j=1}^k \Z^{\overline h_j}\\ \gamma\neq 0} } \left\{\omega(\vect{T})\colon \omega\in \prod_{j=1}^k \Omega_j, \quad\sum_{j=1}^k \mi_{j,\omega_j}(\gamma_j)=0 \right\},
	\]
	and observe that $L$ is one-to-one on $C\setminus N$. In addition, the $\mi_j$ are analytic and homogeneous of homogeneous degree $1$, and the components of the $\Omega_j$ are unbounded; therefore,  $N$ is $\beta_{\Lc'_{A'}}$-negligible.
	Hence, there is a unique $m\colon E_{\Lc_A}\to E_{\Lc'_{A'}}$ such that $m\circ L$ is the identity on $C\setminus N$, while $m$ equals $0$ on the complement of $L(C\setminus N)$. 
	Then, $m$ is $\beta_{\Lc_A}$-measurable, and $\Kc_{\Lc_A}(m)= \Lc'_{A'}\delta_e$.
	Now, let us prove that $m$ is not equal $\beta_{\Lc_A}$-almost everywhere to a continuous function.
	Assume by contradiction that $\Lc'_{A'}\delta_e=\Kc_{\Lc_A}(m')$ for some continuous function $m'$, and let $\pi_0$ be the projection of $G$ onto its abelianization $G'$. 
	Then,~\cite[Proposition 3.2.4]{Martini}, applied to the right quasi-regular representation of $G$ in $L^2(G')$ implies that the operators $\dd\pi_0(\Lc'_1),\dots, \dd\pi_0(\Lc'_k)$ belong to the functional calculus of $\pi_0(\Lc)$, which is absurd.
	
	To conclude, simply take $\tau\in \Sc(E_{\Lc_{A}})$ such that $\tau(\lambda)\neq 0$ for every $\lambda\in E_{\Lc_A}$, and observe that $\Kc_{\Lc_A}(m\tau)= \Lc'_{A'}\Kc_{\Lc_A}(\tau)$ is a family of elements of $\Sc(G)$, while $m\tau$ is not equal $\beta_{\Lc_A}$-almost everywhere to any continuous functions. 
\end{proof}

\subsection{The Case $n_2'<n_2$}

Before we state our main results, let us consider some technical lemmas.

\begin{lem}\label{cor:6}
	Let $M$ be a separable analytic manifold of dimension $n$ endowed with a positive Radon measure $\mi$ which is equivalent to Lebesgue measure on every local chart.
	In addition, take $k,h\in \N$ and a  analytic mapping $P\colon M\to \R^k$ with generic rank $h$ such that $\pi_*(\mi)$ is a Radon measure.  
	Then, the following hold:
	\begin{enumerate}
		\item $P(M)$ is $\Hc^h$-measurable and countably $\Hc^h$-rectifiable;
		
		\item $P_*(\mi)$ is equivalent to $\chi_{P(M)}\cdot \Hc^h$;
		
		\item $\Supp{P_*(\mi)}=\overline{P(M)}$;
		
		\item if $(\beta_y)_{y\in \R^k}$ is a disintegration of $\mi$ relative to $P$, then $\Supp{\beta_y}=P^{-1}(y)$ for $\Hc^h$-almost every $y\in P(M)$.
	\end{enumerate}
\end{lem}

Notice that it is worthwhile for our analysis to consider the case in which $M$ is possibly disconnected. 

\begin{proof}
	Observe first that $M$ may be embedded as a closed submanifold of class $C^\infty$ of $\R^{2 n+1}$ by Whitney embedding theorem (cf.~\cite[Theorem 5 of Chapter 1]{deRham}).
	We may therefore assume that $\mi=f\cdot \Hc^{n}$ for some $f\in L^1_\loc( \chi_M \cdot\Hc^n)$.
	Now,~\cite{Sard} implies that the set where $P$ has rank $<h$, which is $\Hc^n$-negligible by analyticity, has $\Hc^h$-negligible image under $P$.
	Since the image under $P$ of the set where $P$ has rank $h$ is a countable union of analytic submanifolds of $\R^k$ of dimension $h$, we see that $P(M)$ is $\Hc^h$-measurable and countably $\Hc^h$-rectifiable.
	Therefore, we may make use of~\cite[Theorem 3.2.22]{Federer}, and infer that $P_*(\mi)$ is equivalent to the restriction of $\Hc^h$ to the set of $y$ such that $\Hc^{n-h}(P^{-1}(y))>0$, and that we may find a disintegration $(\beta_y)$ of $\mi$ relative to $P$ such that $\beta_y$ is equivalent to $\chi_{P^{-1}(y)}\cdot \Hc^{n-h}$ for $P_*(\beta)$-almost every $y\in \R^k$.
	Now, the preceding arguments show that $P^{-1}(y)$ is an analytic submanifold of dimension $n-h$ of $M$ for $\Hc^h$-almost every $y\in P(M)$. 
	As a consequence, $\Supp{\beta_y}= \Supp{\chi_{P^{-1}(y)}\cdot \Hc^{n-h}}=P^{-1}(y)$ for $\Hc^h$-almost every $y\in P(M)$; for the same reason, we also see that $P_*(\mi)$ is equivalent to $\chi_{P(M)}\cdot \Hc^h$. 
	Finally,  $\Supp{P_*(\mi)}=\overline{P(M)}$ since $P$ is continuous and $\Supp{\mi}=M$.
\end{proof}

\begin{lem}\label{lem:6}
	Let $E_1$, $E_2$ be two finite-dimensional vector spaces, $C$ a convex subset of $E_1$ with non-empty interior, and $L\colon E_1\to E_2$ a linear mapping which is proper on $\partial C$. 
	Assume that for every $x\in \partial C$ either $L^{-1}(L(x))\cap \partial C=\Set{x}$ or $\partial C$ is an analytic hypersurface of $E_1$ in a neighbourhood of $x$.
	Then, $L$ induces an \emph{open} mapping $L'\colon\partial C\to L(\partial C)$.
\end{lem}

\begin{proof}
	Take $x\in \partial C$, and assume that  $L^{-1}(L(x))\cap \partial C=\Set{x}$. 
	Define $U_{x,k}\coloneqq L^{-1}(\overline{B}(L(x), 2^{-k}   ))\cap \partial C$ for every $k\in \N$. 
	Since $L$ is proper on $\partial C$, $U_{x,k}$ is a compact neighbourhood of $x$ for every $k\in \N$. 
	In addition, $\bigcap_{k\in \N} U_{x,k}=\Set{x}$; hence,~\cite[Proposition 1 of Chapter 1, § 9, No.\ 3]{BourbakiGT1} implies that $(U_{x,k})$ is a fundamental system of neighbourhoods of $x$ in $\partial C$, so that $L'$ is open at $x$.
	
	Now, assume that $L^{-1}(L(x))\cap \partial C\neq\Set{x}$. 
	Then, there is a convex neighbourhood $U$ of $L^{-1}(L(x))\cap \partial C$ such that $\partial C \cap U$ is an analytic hypersurface of $E_1$. 
	Assume by contradiction that $\ker L\subseteq T_x(\partial C \cap U) $, and take $x'\in \partial C$ so that $L(x')=L(x)$ but $x'\neq x$. 
	Since $C$ is convex, we have $[x,x']\subseteq \partial C$. 
	Now, $\partial C\cap U$ is an analytic hypersurface and $U$ is convex, so that the arbitrariness of $x'$ implies that $\ell\cap U\subseteq \partial C$, where $\ell$ is the line passing through $x$ and $x'$. 
	Since each $x''\in \ell\cap \partial C$ has a convex neighbourhood where $\partial C$ is an analytic hypersurface, we see that the non-empty closed set $\ell\cap \partial C$ is open in $\ell$. It follows that $\ell\subseteq \partial C$, which is absurd since then $\ell $ would be contained in the compact set $L^{-1}(L(x))\cap \partial C$.
	Therefore, $\ker L\not\subseteq T_x(\partial C \cap U) $, so that $L'$ is open at $x$. 
	The assertion follows.
\end{proof}

\begin{teo}\label{prop:8}
	Assume that $\card(H)=1$ and that $W=\Set{0}$; take a positive integer $n_2'<n_2$. 
	Then, the family $(\Lc, (-i T_j)_{j=1,\dots,n_2'})$ satisfies property $(RL)$.
\end{teo}

\begin{proof}
	{\bf1.} Consider the mapping
	\[
	L\colon E_{\Lc_A}\ni \lambda\mapsto (\lambda_1, (\lambda_{2,j})_{j=1}^{n_2'})\in E_{(\Lc, (-i T_j)_{j=1,\dots,n_2'})};
	\]
	observe that $L(\Lc_A)= (\Lc, (-i T_j)_{j=1,\dots,n_2'})$. Define, in addition, $L'$ in such a way that $L=I_\R\times L'$, and identify $\gf_2^*$ with $\R^{n_2}$ by means of the mapping $\omega \mapsto \omega(\vect{T})$.
	Define, for every $\gamma\in \N^{n_1}$,
	\[
	\beta_\gamma\colon C_c(E_{\Lc_A})\ni \phi\mapsto \int_{\gf_2^*}   \phi( \widetilde \mi_\omega (\vect{1}_{n_1}+2\gamma), \omega)\abs{\Pfaff(\omega)}\,\dd \omega,
	\]
	so that $\beta_{\Lc_A}=\frac{1}{(2\pi)^{n_1+n_2}}\sum_{\gamma\in \N^{n_1}}  \beta_\gamma$. 
	Define
	\[
	\rho_0\colon \R^{n_2'}\ni \omega'\mapsto \min_{ L'(\omega)=\omega' }  \mi_{\omega}(\vect{n_1}),
	\]
	and
	\[
	C_0\coloneqq \Set{ (\rho_0(\omega') +r, \omega'  )\colon \omega'\in \R^{n_2'}, r\Meg 0  },
	\]
	so that $\rho_0$ is a norm on $\R^{n_2'}$ and $C_0=L(\Supp{\beta_0})=L(\sigma({\Lc_A}))$. 
	
	{\bf2.} Now, Corollary~\ref{cor:11:3} implies that $\Supp{\beta_0}\setminus \Set{0}$ is an analytic submanifold of $E_{\Lc_A}$. 
	In addition, Corollary~\ref{cor:6} implies that $L_*(\beta_0)$ is equivalent to $\chi_{\sigma(\Lc,(-i T_j)_{j=1,\dots, n_2'})}\cdot \Hc^{n_2'+1}$ and that, if $(\beta_{0,\lambda})$ is a disintegration of $\beta_{0}$ relative to $L$, then $\Supp{\beta_{0,\lambda}}=L^{-1}(\lambda)\cap \Supp{\beta_0}$ for $L_*(\beta_0)$-almost every $\lambda\in C_0$.
	In addition, Lemma~\ref{lem:7} implies that the mapping $L\colon \Supp{\beta_0}\to C_0$ is open.
	If we prove that $L_*(\beta_\gamma)$ is absolutely continuous with respect to $\Hc^{n_2'+1}$ for every $\gamma\in \N^{n_1}$, the assertion will then follow from Proposition~\ref{prop:A:7}.
	
	Then, let us prove that $L_*(\beta_\gamma)$ is absolutely continuous with respect to $\Hc^{n_2'+1}$ for every $\gamma\in \N^{n_1}$. 
	Notice that this will be the case if we prove that the analytic mapping $\Omega\ni \omega \mapsto  (\widetilde \mi_\omega (\vect{1}_{n_1}+2 \gamma),L'(\omega))$ is generically a submersion for every $\gamma\in \N^{n_1}$ (cf.~Lemma~\ref{cor:6}).
	Assume by contradiction that this is not the case, so that there are $\gamma\in \N^{n_1}$ and a component $C$ of $\Omega$ such that $ \widetilde\mi'_{\omega}(\vect{1}_{n_1}+2\gamma)$ vanishes on $\ker L'$ for every $\omega\in C$.
	As a consequence, there are $(r,\omega')\in \R\times \R^{n_2'}$ such that $L^{-1}(r,\omega')\cap \sigma(\Lc_A)$ contains an open segment. Then, there is a line $\ell$ in $L'^{-1}(\omega')$ such that $(\Set{r}\times \ell)\cap \sigma(\Lc_A) $ contains an open segment; observe that $0\not \in\ell$ since the mapping $\omega \mapsto \widetilde\mi_{\omega}(\vect{1}_{n_1}+2\gamma) $ is homogeneous and proper.
	Then,~\cite[Theorem 6.1 of Chapter II]{Kato} implies that there is an analytic function $f\colon \ell\to \R^{n_1}$ such that $f(\omega)$ is a reordering of $\widetilde \mi_\omega$ for every $\omega\in \ell$. 
	As a consequence, for every $\gamma'\in \N^{n_1}$ the set of $\omega\in \ell$ such that $f(\omega)(\vect{1}_{n_1}+2\gamma')=r$ is compact, hence discrete by analyticity. 
	It is then easily seen that $(\Set{r}\times \ell)\cap \sigma(\Lc_A)$ is  countable, so that it cannot contain any open segments: contradiction.
	The proof is therefore complete.
\end{proof}

\begin{teo}\label{prop:19:2}
	Define  $C_\gamma\coloneqq \Set{( \widetilde\mi_\omega (\vect{1}_{n_1}+2 \gamma),\omega(\vect{T}))\colon \omega\in \gf_2^*}$ for every $\gamma\in \N^{n_1}$. 
	In addition, take $n_2'<n_2$ and define $L\coloneqq\id_\R\times \pr_{1,\dots, n_2'}$ on $E_{\Lc_A}$.
	Assume that the following hold:
	\begin{enumerate}
		\item $\card(H)=1$;
		
		\item $\chi_{C_0}\cdot \beta_{\Lc_A}$ is $L$-connected
		
		\item for every $f\in L^1_{\Lc_A}(G)$ and for every $\gamma\in \N^{n_1}$, $\Mc_{\Lc_A}(f)$ equals $\beta_{\Lc_A}$-almost everywhere a continuous function on $C_\gamma$.
	\end{enumerate}
	Then, $L(\Lc_A)$ satisfies property $(RL)$.
\end{teo}

Observe that condition~{\bf2} holds if $C_0$ is the boundary of a polyhedron (cf.~Corollary~\ref{cor:A:6} below) and if $\Omega =\gf_2^*\setminus \Set{0}$ (cf.~Lemma~\ref{lem:7}). With a little effort, one may prove that condition~{\bf2} holds if $n_2'=1$.

We shall prepare the proof of Theorem~\ref{prop:19:2} through several lemmas.

\begin{lem}\label{lem:8}
	Let $V$ be a topological vector space,  $C$ a convex subset of $V$ with non-empty interior, and $W$ an affine subspace of $V$ such that $W\cap \open C\neq \emptyset$. Then, $W\cap\partial C$ is the frontier of $W\cap C$ in $W$.
\end{lem}

\begin{proof}
	Indeed, take $x_0\in W\cap \open C$, and take $x$ in the interior of $W\cap C$ in $W$. 
	Then, there is $y\in W\cap C$ such that $x\in [x_0,y[$, so that~\cite[Proposition 16 of Chapter II, § 2, No.\ 7]{BourbakiTVS} implies that $x\in \open C$.
	By the arbitrariness of $x$, this proves that $W\cap \open C$ is the interior of $W\cap C$ in $W$.  
	Analogously, one proves that $W\cap \overline C$ is the closure of $W\cap C$ in $W$, whence the result. 
\end{proof}

\begin{lem}\label{lem:4}
	Let $f\colon \R^n\to \R$ be a convex function which is differentiable on an open subset $U$ of $\R^n$. 
	Let $L$ be a linear mapping of $\R^n$ onto $\R^k$ for some $k\meg n$, and assume that $(f,L)$ has rank $k$ on $U$. 
	Then, for every $y\in (f,L)(U)$, the fibre $(f,L)^{-1}(y)$ is a closed convex set which contains $L^{-1}(y_2)\cap U$.
\end{lem}

\begin{proof}
	Define $\pi\coloneqq (f,L)\colon \R^n\to \R\times \R^k$, and observe that $\ker L \subseteq \ker f'(x)$ for every $x\in U$ since $\pi $ has rank $k$ on $U$. 
	Therefore, if $y\in \pi (U)$, then $f$ is locally constant on $L^{-1}(y_2)\cap U$. 
	Now, take two components $C_1$ and $C_2$ of $L^{-1}(y_2)\cap U$, and observe that they are open in $L^{-1}(y_2)$. 
	Take $x_1\in C_1$ and $x_2\in C_2$. 
	Then $[x_1,x_2]\subseteq L^{-1}(y_2)$, so that there are $x_1',x_2'\in ]x_1,x_2[$ such that $f$ is constant on $[x_1,x_1']$ and on $[x_2',x_2]$. 
	By convexity, $f$ must be constant on $[x_1,x_2]$, hence on $C_1\cup C_2$. 
	By the arbitrariness of $C_1$ and $C_2$, we infer that  $\pi^{-1}(y)\supseteq L^{-1}(y_2)\cap U$.

	Now, consider the closed convex set $C\coloneqq \Set{(\lambda,x)\colon x\in \R^n, \lambda \Meg f(x)}$, and observe that 
	\[
	\open C=  \Set{(\lambda,x)\colon x\in \R^n, \lambda >f(x)}
	\]
	since $f$ is continuous, so that $\partial C$ is the graph of $f$. 
	Next, define $W\coloneqq (\id_\R\times L)^{-1}(y)=\Set{y_1}\times L^{-1}(y_2)$, and observe that  $W\cap \partial C=\Set{y_1}\times\pi^{-1}(y)$.
	Assume by contradiction that $W\cap \open C\neq \emptyset$. 
	Then,  Lemma~\ref{lem:8} implies that $W\cap \partial C$ is the frontier of $W\cap C$ in $W$, so that $\pi^{-1}(y)$ has empty interior in $L^{-1}(y_2)$. 
	However, $\pi^{-1}(y)$ contains $L^{-1}(y_2)\cap U$, which is open in $ L^{-1}(y_2)$: contradiction. 
	Therefore, $\Set{y_1}\times\pi^{-1}(y)= W\cap C$ is a closed convex set, whence the result.
\end{proof}

\begin{lem}\label{lem:5}
	Let $f\colon \R^n\to \R$ be a convex function which is analytic on some  open subset $\Omega$ of $\R^n$ whose complement is $\Hc^n$-negligible. 
	Let $L$ be a linear mapping of $\R^n$ onto $\R^k$ for some $k\meg n$, and let $U$ be the union of the components of $\Omega $ where $(f,L)$ has rank $k$.
	Then, 
	\[
	(f,L)^{-1}(y)= \overline{L^{-1}(y_2)\cap U}
	\]
	for $\Hc^k$-almost every $y\in (f,L)(U)$.
\end{lem}

\begin{proof}
	Define $\pi\coloneqq (f,L)$.
	Since the complement of $\Omega $ is $\Hc^n$-negligible, there is an $\Hc^k$-negligible subset $N_1$ of $\R^k$ such that $L^{-1}(y)\setminus \Omega $ is $\Hc^{n-k}$-negligible for every $y\in \R^k\setminus N_1$ (cf.~\cite[Theorem 3.2.11]{Federer}).
	In addition, observe that the set $R_k$ of $x\in \Omega \setminus U$ such that $\ker L\subseteq \ker f'(x)$, that is, such that $\pi '(x)$ has rank $k$, is $\Hc^n$-negligible by the analyticity of $f$. 
	Then, there is an $\Hc^k$-negligible subset $N_2$ of $\R^k$ such that $L^{-1}(y)\cap \Omega \setminus U$ is $\Hc^{n-k}$-negligible for every $y\in \R^k\setminus N_1$ (\emph{loc.\ cit.}).
	Now, define $N\coloneqq \R \times (N_1 \cup N_2)$, and observe that $\chi_{\pi (U)}\cdot \Hc^k$ is equivalent to the (not necessarily Radon) measure $\pi_*(\chi_U\cdot \Hc^n)$  thanks to Corollary~\ref{cor:6}; since $U\setminus \pi^{-1}(N)=U\setminus L^{-1}(N_1 \cup N_2)$ is $\Hc^{n}$-negligible, it follows that $\pi (U)\cap N$ is $\Hc^k$-negligible.
	
	Now, take $y\in \pi (U)\setminus N$. 
	Then, Lemma~\ref{lem:4} implies that $\pi^{-1}(y)$ is a closed convex set which contains $L^{-1}(y_2)\cap U$, so that its interior in $L^{-1}(y_2)$ is not empty.
	Let $U'$ be a component of $\Omega $ which is not contained in $U$, and assume that $\pi^{-1}(y)\cap U'\neq \emptyset$.  
	Since $f$ is analytic on $U'$, and since $\pi^{-1}(y)$ is a convex set with non-empty interior in $L^{-1}(y_2)$, we see that a component $C$ of  $L^{-1}(y_2)\cap U'$ is contained in $ R_k$.  
	By the choice of $N_2$, this implies that $C $ is $\Hc^{n-k}$-negligible; since $C$ is non-empty and open in $L^{-1}(y_2)$, this leads to a contradiction. 
	Therefore, 
	\[
	L^{-1}(y_2)\cap U\subseteq\pi^{-1}(y) \subseteq   L^{-1}(y_2)\cap [U \cup (\R^n\setminus \Omega )].
	\]
	By our choice of $N_1$, the set $L^{-1}(y_2) \setminus \Omega $ is $\Hc^{n-k}$-negligible; on the other hand, the support of $\chi_{\pi^{-1}(y)}\cdot \Hc^{n-k}$ is $\pi^{-1}(y)$ by convexity. 
	Hence, $L^{-1}(y_2)\cap U$ is dense in $\pi^{-1}(y) $, whence the result.
\end{proof}

\begin{lem}\label{lem:7}
	Keep hypotheses and notation of Lemma~\ref{lem:5}. Assume, in addition, that $\lim\limits_{x\to \infty} f(x)=+\infty$ and that $\Hc^n$ is $(f,L)$-connected.
	Then, for every $m\in C(\R^n)$ such that $m=m'\circ L$ $\Hc^n$-almost everywhere for some $m'\colon \R\times \R^k\to \C$, there is $m''\in C(\R\times \R^k)$ such that $m=m''\circ (f,L)$ pointwise. 
\end{lem}

\begin{proof}
	Define $\pi\coloneqq (f,L)$.
	Let $(\beta_{1,y})_{y\in \R \times \R^k}$ be a disintegration of $\chi_U\cdot\Hc^n$ relative to $\pi$ and let $(\beta_{2,y})_{y\in \R \times \R^k}$ be a disintegration of $\chi_{\Omega \setminus U}\cdot\Hc^n$ relative to $\pi$. 
	Then, Corollary~\ref{cor:6} implies that:
	\begin{itemize}
		\item $\pi_* (\chi_U\cdot \Hc^n)$ is equivalent to $\chi_{\pi (U)}\cdot \Hc^k$;
		
		\item $\pi_* (\chi_{\Omega \setminus U}\cdot \Hc^n)$ is equivalent to $\chi_{\pi (\Omega \setminus U)}\cdot \Hc^{k+1}$;
		
		\item $\Supp{\beta_{1,y}}=\overline{ \pi^{-1}(y)\cap U }$ for $\Hc^k$-almost every $y\in \pi (U)$;
		
		\item $\Supp{\beta_{2,y}}=\overline{ \pi^{-1}(y)\cap \Omega \setminus U }$ for $\Hc^{k+1}$-almost every $y\in \pi (\Omega \setminus U)$.
	\end{itemize} 
	In addition, $\pi(U)$ has Hausdorff dimension $k$, so that $\Hc^{k+1}(\pi(U))=0$; in particular, $\pi_* (\chi_U\cdot \Hc^n)$ and $\pi_* (\chi_{\Omega \setminus U}\cdot \Hc^n)$ are alien measures.
	If we define $\beta_y\coloneqq \beta_{1,y}$ for every $y\in \pi(U)$ and $\beta_y\coloneqq \beta_{2,y}$ for every $y\in (\R \times \R^k)\setminus \pi(U)$, then $(\beta_y)$ is a disintegration of $\Hc^n$ relative to $\pi$.
	
	Now, Lemma~\ref{lem:5} implies that $\overline{ \pi^{-1}(y)\cap U }=\pi^{-1}(y) $ for $\Hc^k$-almost every $y\in \pi(U)$.
	Next, let us prove that $\pi^{-1}(y)=\overline{ \pi^{-1}(y)\cap \Omega \setminus U }$ for $\Hc^{k+1}$-almost every $y\not\in \pi(U)$. 
	Let us first prove that $ \pi^{-1}(y)$ is the boundary of a compact convex set with non-empty interior in $L^{-1}(y_2)$ for $\Hc^{k+1}$-almost every $y\in \pi(\Omega \setminus U)$. 
	
	Indeed, by~\cite{Sard} there is an $\Hc^{k+1}$-negligible subset $N$ of $\pi(\Omega \setminus U)$ such that $\pi'(x)$ has rank $k+1$ for every $x\in L^{-1}(y)\cap \Omega \setminus U$ and for every $y\in \pi(\Omega \setminus U)\setminus N$. 
	Now, define $C\coloneqq \Set{(\lambda, x)\colon x\in \R^n, \lambda \Meg f(x)}$, and observe that $\id_\R\times L$ is proper on $C$ since $\lim_{x\to \infty} f(x)=+\infty$. 
	Therefore, $\Set{y_1}\times \pi^{-1}(y)=( \id_\R\times L)^{-1}(y)\cap \partial C$ is compact for every $y\in \R\times \R^k$. 
	In addition, if $y\in \pi(\Omega \setminus U)\setminus N$, then $( \id_\R\times L)^{-1}(y)\cap \open C\neq \emptyset$, so that Lemma~\ref{lem:8} implies that $\pi^{-1}(y)$ is the boundary of a compact convex set with non-empty interior in $L^{-1}(y_2)$.

	Therefore, $ \pi^{-1}(y)$ is bi-Lipschitz homeomorphic to $\Sd^{n-k-1}$,
	so that the support of $\chi_{\pi^{-1}(y)}\cdot \Hc^{n-k-1}$ is $\pi^{-1}(y)$ for such $y$. 
	In addition, since $\R^n\setminus \Omega $ is $\Hc^n$-negligible,~\cite[Theorem 3.2.11]{Federer} implies that $\pi^{-1}(y)\setminus \Omega $ is $\Hc^{n-k-1}$-negligible for $\Hc^{k+1}$-almost every $ y\in \R\times \R^k$.
	Hence, $\overline{ \pi^{-1}(y)\cap \Omega \setminus U }= \pi^{-1}(y)$ for $\Hc^{k+1}$-almost every $y\not \in \pi(U)$. 
	
	Then, Proposition~\ref{prop:A:7} implies that there is $m'''\colon\pi(\R^n)\to \C$ such that $m=m'''\circ \pi$; since $\pi$ is proper, this implies that $m'''$ is continuous on $\pi(\R^n)$. 
	Finally, since $\pi$ is proper, $\pi(\R^n)$ is closed, so that the assertion follows from~\cite[Corollary to Theorem 2 of Chapter IX, § 4, No.\ 2]{BourbakiGT2}.
\end{proof}

\begin{proof}[Proof of Theorem~\ref{prop:19:2}]
	Until the end of this proof, we shall identify $\R^{n_2}$ and $\gf_2^*$ by means of the bijection $\omega \mapsto \omega(\vect{T})$; $L'$ will denote $\pr_{1,\dots,n_2'}$, so that $L=\id_\R\times L'$.
	In addition, for every $\gamma\in \N^{n_1}$, define $\pi_\gamma\colon \R^{n_2}\ni \omega \mapsto (\mi_\omega(\vect{1}_{n_1}+2\gamma),\omega)$, so that $\pi_\gamma$ is continuous and $C_\gamma$ is the graph of $\pi_\gamma$.
	
	Take $f\in L^1_{L(\Lc_A)}(G)$ and let $m$ be a representative of $\Mc_{L(\Lc_A)}(f)$. 
	Take, for every $\gamma\in \N^{n_1}$, a continuous function $m_\gamma$ on $C_\gamma$ such that $m_\gamma=\Mc_{\Lc_A}(f)$ $\chi_{C_\gamma}\cdot \beta_{\Lc_A}$-almost everywhere.
	Then, Lemma~\ref{lem:7} implies that there is a continuous function $m'_0\colon E_{L(\Lc_A)}\to \C$ such that $m_0=m'_0\circ L $ \emph{on $C_0$}.
	Since $\beta_{L(\Lc_A)}$ need \emph{not} be equivalent to $L_*(\chi_{C_0}\cdot \beta_{\Lc_A})$, though, this is not sufficient to conclude.
	
	For every $\gamma\in \N^{n_1}$, define $\beta_\gamma\coloneqq \chi_{C_\gamma}\cdot \beta_{\Lc_A}$, and let $U_{\gamma,1}$ be the union of the components $C$ of $\Omega$ such that $\widetilde \mi'_\omega(\vect{1}_{n_1}+2 \gamma)$ does not vanish on $\ker L'$ for some $\omega\in C$. 
	Let $U_{\gamma,2}$ be the complement of $U_{\gamma,1}$ in $\Omega$.
	Notice that $\beta_\gamma$ is equivalent to $(\pi_\gamma)_*(\Hc^{n_2})$.
	In addition, Corollary~\ref{cor:6} implies that the following hold:
	\begin{itemize}
		\item $L_*(\chi_{\R\times U_{\gamma,1}}\cdot \beta_\gamma)$ is equivalent to $\chi_{L(\pi_\gamma(U_{\gamma,1}))}\cdot \Hc^{n_2'+1}$;
		
		\item $L_*(\chi_{\R\times U_{\gamma,2}}\cdot \beta_\gamma)$ is equivalent to  $\chi_{L(\pi_\gamma( U_{\gamma,2}))}\cdot \Hc^{n_2'}$;
		
		\item $\chi_{\R\times U_{\gamma,2}}\cdot \beta_\gamma$ has a disintegration $(\beta_{\gamma,2,\lambda})_{\lambda\in E_{L(\Lc_A)}}$ relative to $L$ such that $L^{-1}(\lambda)\cap\pi_\gamma(U_{\gamma,2})\subseteq\Supp{\beta_{\gamma,2,\lambda}}$ and $\beta_{\gamma,2,\lambda}$ is equivalent to the measure $\chi_{L^{-1}(\lambda)\cap \pi_\gamma( U_{\gamma,2})}\cdot \Hc^{n_2-n_2'}$ for $\Hc^{n_2'}$-almost every $\lambda\in L(\pi_\gamma(U_{\gamma,2}))$.
	\end{itemize}
	In particular, $\beta_{L(\Lc_A)}$ is equivalent to $\chi_{\sigma(L(\Lc_A))}\cdot \Hc^{n_2'+1}+\mi$, where $\mi $ is a measure absolutely continuous with respect to $\Hc^{n_2'}$ alien to $\Hc^{n_2'+1}$.
	Now, observe that  $L_*(\chi_{\R\times U_{\gamma,1}}\cdot \beta_\gamma)$ is absolutely continuous with respect to $L_*(\beta_0)$; since $(m-m_0')\circ L$ is $\beta_0$-negligible, there is an $L_*(\beta_0)$-negligible subset $N$ of $E_{L(\Lc_A)}$ such that $m=m_0'$ on $E_{L(\Lc_A)}\setminus N$. 
	Since $N$ is  $L_*(\chi_{\R\times U_{\gamma,1}}\cdot \beta_\gamma)$-negligible, this implies that $(m-m_0')\circ L$ vanishes $\chi_{\R\times U_{\gamma,1}}\cdot \beta_\gamma$-almost everywhere. 
	Since $m\circ L=m_\gamma$ $\beta_\gamma$-almost everywhere, it follows that $m_0'\circ L=m_\gamma$ $ \chi_{\R\times U_{\gamma,1}}\cdot \beta_\gamma$-almost everywhere, hence on 
	\[
	\Supp{\chi_{\R\times U_{\gamma,1}}\cdot \beta_\gamma}=\Supp{(\pi_\gamma)_*(\chi_{U_{\gamma,1}}\cdot \Hc^{n_2})}= \pi_\gamma(\overline{U_{\gamma,1}}),
	\] 
	since $m_0'\circ L$ and $m_\gamma$ are continuous, while $\pi_\gamma$ is proper.

	Next, consider $ \chi_{\R\times U_{\gamma,2}}\cdot \beta_\gamma$. 
	Tonelli's theorem implies that  $L'^{-1}(\lambda_2)\setminus \Omega$ is $\Hc^{n_2-n_2'}$-negligible for $\Hc^{n_2'}$-almost every $\lambda_2\in \R^{n_2'}$. 
	Now, if $\widetilde N$ is an $\Hc^{n_2'}$-negligible subset of $ \R\times \R^{n_2'}$, then $\pr_2(\widetilde N)$ is $\Hc^{n_2'}$-negligible since $\pr_2$ is Lipschitz. 
	Therefore, there is an $\Hc^{n_2'}$-negligible subset $N'$ of $\R^{n_2'}$ such that, for every $\lambda\in L(\pi_\gamma(U_{\gamma,2}))\setminus (\R\times N')$,
	\begin{itemize}
		\item $m\circ L=m_\gamma$ $\beta_{\gamma,2,\lambda}$-almost everywhere;
		
		\item $L^{-1}(\lambda)\cap\pi_\gamma( U_{\gamma,2})\subseteq\Supp{\beta_{\gamma,2,\lambda}} $;
		
		\item $L'^{-1}(\lambda_2) \setminus \Omega$ is $\Hc^{n_2-n_2'}$-negligible.
	\end{itemize}
	Hence, if $\lambda\in L(\pi_\gamma(U_{\gamma,2}))\setminus (\R\times N')$, then $m_\gamma$ is constant on $L^{-1}(\lambda) \cap \pi_\gamma(U_{\gamma,2})$.
	In addition, fix $\lambda\in L(\pi_\gamma(U_{\gamma,2}))\setminus (\R\times N')$; then,
	\[
	L'^{-1}(\lambda_2)=\overline{L'^{-1}(\lambda_2)\cap U_{\gamma,1}}\cup \overline{L'^{-1}(\lambda_2)\cap U_{\gamma,2}},
	\]
	so that $\overline{L'^{-1}(\lambda_2)\cap U_{\gamma,1}}\cap \overline{L'^{-1}(\lambda_2)\cap U_{\gamma,2}}\neq \emptyset$ or $L'^{-1}(\lambda_2)\cap U_{\gamma,1}= \emptyset$ by connectedness. 
	
	Now, let $\Cc$ be the set of components of $L'^{-1}(\lambda_2)\cap U_{\gamma,2}$; observe that $\Cc$ is finite since $L'^{-1}(\lambda_2)\cap \Omega$ is semi-algebraic (cf.~\cite[Proposition 4.13]{Coste}) and since $L'^{-1}(\lambda_2)\cap U_{\gamma,2}$ is open and closed in $L'^{-1}(\lambda_2)\cap \Omega$.
	In addition, observe that $\pr_1\circ \pi_\gamma$ is constant on each $C\in \Cc$; let $\lambda_{1,C}$ be its constant value. 
	In particular, since $\pr_1\circ \pi_\gamma$ is proper and since $\Cc$ is finite, this implies that $L'^{-1}(\lambda_2)\cap U_{\gamma,1}\neq \emptyset$.
	Further, $m_\gamma$ is constant on $\pi_\gamma( C)\subseteq L^{-1}(\lambda_{1,C},\lambda_2)\cap \pi_\gamma(U_{\gamma,2})$ for every $C\in \Cc$.
	Now, there is $C_1\in \Cc$ such that $L'^{-1}(\lambda_2)\cap \overline{U_{\gamma,1}}\cap \overline{C_1}\neq \emptyset$; since $m_\gamma\circ \pi_\gamma=m_0'\circ L\circ \pi_\gamma$ on $\overline{U_{\gamma,1}}$, and since $m_\gamma$ is continuous, it follows that $m_\gamma\circ \pi_\gamma=m_0'\circ L\circ \pi_\gamma$ on $\overline C_1$.  
	Iterating this procedure, we eventually see that $m_\gamma\circ \pi_\gamma=m_0'\circ L\circ \pi_\gamma$ on $L'^{-1}(\lambda_2)$.
	Therefore, $m_\gamma=m_0'\circ L$ on $L^{-1}(\lambda)\cap C_\gamma$ for every $\lambda \in L(\pi_\gamma(U_{\gamma,2}))\setminus (\R\times N')$.
	
	Now, observe that $L^{-1}(\R\times N')\cap \pi_\gamma(U_{\gamma,2})$ is $\Hc^{n_2'}$-negligible since $\pr_2\circ L\circ \pi_\gamma=L'$ and since $\Hc^{n_2''}$ is equivalent to the non-Radon measure $L'_*(\Hc^{n_2})$.
	Therefore, $m_\gamma=m_0'\circ L$ $\beta_\gamma$-almost everywhere, hence on $C_\gamma$ by continuity.
	By the arbitrariness of $\gamma$, this implies that $m_0'\circ L$ is a representative of $\Mc_{\Lc_A}(f)$, so that $m_0'$ is a continuous representative of $\Mc_{L(\Lc_A)}(f)$. The assertion follows.	
\end{proof}

\section{Property $(S)$}\label{sec:6}

The results of this section are basically a generalization of the techniques employed in~\cite{AstengoDiBlasioRicci,AstengoDiBlasioRicci2}.
The first result has very restrictive hypotheses, for the same reasons explained while discussing property $(RL)$, but hold for the `full family' $\Lc_A$ (cf.~Theorem~\ref{prop:20:5}); on the contrary, the second one  holds under more general assumptions, but only for families of the form $(\Lc,(-i T_1,\dots, - i T_{n_2'}))$ for $n_2'<n_2$ (cf.~Theorem~\ref{prop:20:4}).

Notice that, even though Theorem~\ref{teo:15:1} is the main application of Theorem~\ref{prop:20:5}, there are other families to which it applies as well. This happens for the family we considered while discussing property $(RL)$ in the case of Theorem~\ref{prop:19:3}.

Notice that in all of our results we imposed the condition $W=\Set{0}$; this is unavoidable (with our methods), since on $W$ we cannot infer any kind of regularity from the `inversion formulae' employed. Indeed, our auxiliary function $\abs{x_\omega}^2$ is not differentiable on $W$, in general.
Nevertheless, this does not mean that property $(S)$ cannot hold when $W\neq \Set{0}$, as Theorem~\ref{prop:13} shows.

Before stating our first result, let us recall a lemma based on some techniques developed in~\cite{Geller} and then in~\cite{AstengoDiBlasioRicci}.

\begin{lem}[\cite{Calzi}, Lemma 11.1]\label{lem:20:3}\label{lem:20:4}
	Let $\Lc_A$ be a Rockland family on a homogeneous group $G'$, and let $T'_1,\dots, T'_n$ be a free family of elements of the centre of the Lie algebra $\gf'$ of $G'$. 
	Let $\pi_1$ be the canonical projection of $G'$ onto its quotient by the normal subgroup $\exp(\R T'_1)$, and assume that the following hold:
	\begin{itemize}
		\item $(\Lc_A, i T'_1,\dots, i T'_n )$ satisfies property $(RL)$;
		
		\item  $\dd \pi_1(\Lc_A, i T'_{2},\dots, i T'_n)$ satisfies property $(S)$.
	\end{itemize}
	
	Take $\phi\in \Sc_{(\Lc_A, i T'_1,\dots, i T'_n )}(G')$. 
	Then, there are two families $(\widetilde \phi_\gamma)_{\gamma\in \N^{n}}$ and $(\phi_\gamma)_{\gamma\in \N^{n}}$  of elements of $\Sc(G',\Lc_A)$ and $\Sc_{(\Lc_A, i T'_1,\dots, i T'_n )}(G')$, respectively, such that 
	\[
	\phi=\sum_{\abs{\gamma}<h} \vect{T}'^\gamma \widetilde \phi_\gamma+\sum_{\abs{\gamma}=h} \vect{T}'^\gamma  \phi_\gamma
	\]
	for every $h\in \N$.
\end{lem}

\begin{teo}\label{prop:20:5}
	Assume that $\Omega =\gf_2^*\setminus \Set{0}$, that $\dim_\Q \mi_\omega (\Q^{\overline h})=\dim_\R \mi_\omega (\R^{\overline h})$ for every $\omega\in \Omega$, and that $\mi$ is constant where $\mi_{\eta _0}(\vect{n_1})$ is constant. 
	Then, $\Lc_A$ satisfies property $(S)$.
\end{teo}

\begin{proof}
	We proceed by induction on $n_2\Meg 1$.
	
	{\bf1.} Notice that the inductive hypothesis, Theorem~\ref{teo:9:3}, Theorem~\ref{prop:19:3}, and Lemma~\ref{lem:20:4} imply that we may find a family $(\widetilde \phi_\gamma)$ of elements of $\Sc(G,\Lc_H)$, and a family $(\phi_\gamma)$ of elements of $\Sc_{\Lc_A}(G)$ such that
	\[
	\phi=\sum_{\abs{\gamma}< h} \vect{T}^\gamma\widetilde \phi_\gamma+ \sum_{\abs{\gamma}=h}\vect{T}^\gamma \phi_\gamma
	\] 
	for every $h\in \N$.
	
	Define $\widetilde m_\gamma\coloneqq \Mc_{\Lc_H}(\widetilde \phi_\gamma)\in \Sc(\sigma(\Lc_H))$ and $m_\gamma\coloneqq \Mc_{\Lc_A}(\phi_\gamma)\in C_0(\beta_{\Lc_A})$ for every $\gamma$ (cf.~Theorem~\ref{prop:19:3}).
	Then, 
	\[
	m_0(\lambda, \omega) =\sum_{\abs{\gamma}<h} \omega^\gamma \widetilde m_\gamma(\lambda)+ \sum_{\abs{\gamma}=h} \omega^\gamma m_\gamma(\lambda, \omega)
	\]
	for every $h\in \N$ and for every $(\lambda, \omega)\in \sigma(\Lc_A)$. 
	
	{\bf2.} Assume  that $m_\gamma=0$ for every $\gamma\in \N^{n_2}$.
	Define $N(\omega)\coloneqq \mi_\omega(\vect{n_1})$ for every $\omega\in \gf_2^*$, so that $N$ is a (homogeneous) norm on $\gf_2^*$ which is analytic on $\omega$. 
	Define, in addition, $\Sigma\coloneqq \mi_\omega(\vect{n_1}+2 \N^{\overline h})$ for some (hence every) $\omega\in \gf_2^*$ such that $N(\omega)=1$. 
	Then, set $d\coloneqq \inf_{\sigma\in \Sigma}d\left( \sigma, \Sigma\setminus\Set{\sigma}  \right)$, and observe that $d>0$ since $\dim_\Q \mi_\omega(\Q^{\overline h})=\dim_\R \mi_\omega(\R^{\overline h})$. 
	Finally, identify $ \gf_2^*$ with $\R^{n_2}$ by means of the mapping $\omega \mapsto \omega(\vect{T})$, take $r\in \left]0,\frac{\min_{\sigma\in\Sigma} \abs{\sigma}}{4 d}\right[$, and choose
	$\phi\in \Dc(\R^{H})$ so that $\chi_{B(0,r)}\meg \phi \meg \chi_{B(0,2 r)}$.  Define
	\[
	\widetilde m(\lambda)\coloneqq \begin{cases}
	\sum_{\sigma\in \Sigma} \widetilde m_0( N(\lambda_2) \sigma, \lambda_2) \phi\left( \frac{1}{d}\left(\frac{\lambda_1}{N(\lambda_2)} -\sigma \right) \right) & \text{if $\lambda_2\neq 0$}\\
	0 & \text{if $\lambda_2=0$}
	\end{cases}
	\]
	for every $\lambda\in E_{\Lc_{A}}$. 
	Proceeding as in the proof of~\cite[Lemma 3.1]{AstengoDiBlasioRicci}, one sees that $\widetilde m\in \Sc(E_{\Lc_A})$, so that $\phi\in \Sc(G,\Lc_A)$.

	{\bf3.} Now, consider the general case. 
	By a vector-valued version of Borel's lemma (cf.~\cite[Theorem 1.2.6]{Hormander2} for the scalar, one-dimensional case),  there is $\widehat m\in \Dc(\gf_2^*; \Sc(\R^H) )$ such that $\widehat m^{(\gamma)}(0)=\widetilde  m_\gamma$ for every $\gamma\in \N^{n_2}$. 
	Interpret $\widehat m$ as an element of $\Sc(E_{\Lc_A})$.
	Then,~{\bf2} implies that $m-\widehat m$ equals a Schwartz function on $\sigma(\Lc_A)$. The assertion follows. 
\end{proof}

Now we consider the case in which $\card(H)=1$, and $n_2'<n_2$. We begin with a suitable version of Morse lemma, which is an easy consequence of~\cite[Lemma C.6.1]{Hormander1}.

\begin{lem}\label{lem:20:1}
	Let $U$ an open subset of $\R^k\times \R^n$,  and $\phi$ a mapping of class $C^\infty$ of $U$ into $\R$.
	Assume that $\partial_1\phi(x_0)=0$ and that $\partial_1^2 \phi(x_0)$ is positive and non-degenerate for some $x_0\in U$.
	
	Then, there are an open neighbourhood $V_1$ of $0$ in $\R^k$, an open neighbourhood $V_2$ of $x_{0,2}$ in $\R^n$, and a $C^{\infty}$-diffeomorphism $\psi$ from $V_1 \times V_2$ onto an open subset of $U$ such that $\psi(0,x_{0,2})=x_0$, $\psi_2=\pr_2$, and
	\[
	\phi(\psi(y))= \phi(\psi(0,y_2))+ \norm{y_1}^2
	\]
	for every $y\in V_1\times V_2$.
\end{lem}

\begin{cor}\label{cor:20:1}
	Keep the hypotheses and the notation of Lemma~\ref{lem:20:1}.
	Take a function $f\in C^\infty(\psi(V_1\times V_2)\times \R)$ and a function $g\colon V_2\times \R\to \C$ so that
	\[
	f(x, \phi(x))= g(x_2, \phi(x))
	\]
	for every $x\in \psi(V_1\times V_2)$. 
	Then, $g$ can be modified so as to be of class $C^\infty$ in a neighbourhood of $(x_{0,2}, \phi(x_0))$.
\end{cor}

\begin{proof}
	Indeed, the assumption means that
	\[
	f\left(y, \phi(\psi(0,y_2))+\norm{y_1}^2\right)=g\left(y_2, \phi(\psi(0,y_2))+\norm{y_1}^2\right) 
	\]
	for every $y\in V_1\times V_2$.
	Define, for every $y_2\in V_2$, 
	\[
	\widetilde f_{y_2}\colon V_1\ni y_1\mapsto f((y_1,y_2), \phi\left(\psi(0,y_2))+\norm{y_1}^2\right) \qquad \text{and} \qquad
	\widetilde g_{y_2}\colon  \R \ni t\mapsto g(y_2, \phi(\psi(0,y_2))+t).
	\]
	Then, the mapping $V_2\ni y_2 \mapsto \widetilde f_{y_2}$ belongs to $\Ec(V_2; \Ec(V_1))$, and
	\[
	\widetilde f_{y_2}(y_1)= \widetilde g_{y_2}\left(\norm{y_1}^2\right)
	\]
	for every $y_1\in V_1$ and for every $y_2\in V_2$.
	
	Now,~\cite{Whitney2} easily implies that the mapping\footnote{We denote by $\Ec_\R(\R_+)$ the quotient of $\Ec(\R)$ by the set of $\phi\in \Ec(\R)$ which vanish on $\R_+$.}
	\[
	\Phi_1\colon\Ec_\R(\R_+) \ni h \mapsto h\circ \norm{\,\cdot\,}^2\in \Ec(\R^k)
	\]
	is an isomorphism onto the set of radial functions of class $C^\infty$. 
	Since there is a continuous linear extension operator $\Ec_\R(\R_+)\to \Ec(\R)$ (cf., for instance,~\cite[Corollary 0.3]{BierstoneMilman}), we  find a continuous linear mapping $\Phi_2\colon \Phi_1(\Ec_\R(\R_+))\to \Ec(\R)$ such that
	\[
	\Phi_2(h)\circ \norm{\,\cdot\,}^2=h
	\]
	for every even function $h\in \Ec(\R)$. 
	Then, take $\tau\in \Dc(V_1)$ so that $\tau$ equals $1$ on a neighbourhood $V_1'$ of $0$ in $V_1$, and define
	\[
	\widetilde G_{y_2}\colon V_1 \ni t \mapsto \Phi_2(\tau \widetilde f_{y_2} ) (t).
	\]
	Then, $\widetilde G_{y_2}\left(\norm{y_1}^2\right)= \widetilde g_{y_2}\left(\norm{y_1}^2\right)$ for every $y_1\in V_1'$ and for every $y_2\in V_2$.
	In addition, the mapping $y_2 \mapsto \widetilde G_{y_2}$ belongs to $\Ec(V_2; \Ec(\R))$, so that there is $\widetilde G\in \Ec(V_2\times \R)$ such that $\widetilde G(y_2,t)= \widetilde G_{y_2}(t)$ for every $y_2\in V_2$ and for every $t\in \R$.
	Then, 
	\[
	g\left(y_2, \phi(\psi(0,y_2))+\norm{y_1}^2\right) =\widetilde G\left(y_2, \norm{y_1}^2\right)
	\]
	for every $y_2\in V_2$ and for every $y_1\in V_1'$.
	Define
	\[
	G\colon V_2 \times \R\ni (y_2, t)\mapsto \widetilde G(y_2, t-\phi(\psi(0,y_2))),
	\]
	so that $G\in \Ec(V_2\times \R)$ and
	\[
	f(x, \phi(x))= G(x_2, \phi(x))
	\]
	for every $x\in \psi(V_1'\times V_2)$, whence the result. 
\end{proof}

\begin{teo}\label{prop:20:4}
	Assume that $\card(H)=1$ and that $W=\Set{0}$, and let $S$ be the analytic hypersurface $\Set{\omega\in \gf_2^*\colon \mi_\omega(\vect{n_{1,\omega}})=1 }$. Assume that, for every $\omega\in S$ such that $\langle T_1,\dots, T_{n_2'}\rangle^\circ\subseteq T_\omega(S)$, the Gaussian curvature of $S$ at $\omega$ is non-zero.
	Take $n_2'\in \Set{0,\dots, n_2-1}$ and define  $\Lc_{A'}\coloneqq (\Lc, (-i T_1,\dots, - i T_{n_2'}))$.
	Then, $\Lc_{A'}$ satisfies property $(S)$.
\end{teo}

The condition on $S$ is satisfied, for example, if $\omega \mapsto \mi_\omega(\vect{n_{1,\omega}})$ is a hilbertian norm.
Observe, in addition, that the Gaussian curvature of $S$ vanishes on a negligible set in virtue of the strict convexity of the norm  $\omega \mapsto \mi_\omega(\vect{n_{1,\omega}})$. 
Therefore, for almost every $(T'_1,\dots, T'_{n_2'})\in \gf_2^{n_2'}$ the family $(\Lc, (- i T'_1,\dots, - i T'_{n_2'}))$ satisfies property $(S)$.

\begin{proof}
	{\bf1.} We proceed by induction on $n_2'$. Observe first that  the assertion follows from Theorem~\ref{teo:9:3} when $n_2'=0$. Then, assume that $n_2'>0$. By the inductive assumption, it is easily seen that $\dd \pi_1(\Lc_{A'})$ satisfies property $(S)$, where $\pi_1$ is the canonical projection of $G$ onto $\quot{G}{\exp_G(\R T_1)}$.
	
	Take $\phi\in \Sc_{\Lc_{A'}}(G)$.
	Then, 
	Theorem~\ref{prop:8} and Lemma~\ref{lem:20:3} imply that we may find a family $(\widetilde \phi_\gamma)_{\gamma\in \N^{n_2'}}$ of elements of $\Sc(G,\Lc)$, and a family $(\phi_\gamma)_{\gamma\in \N^{n_2'}}$ of elements of $\Sc_{\Lc_{A'}}(G)$ such that
	\[
	\phi=\sum_{\abs{\gamma}< h} \vect{T}^\gamma\widetilde \phi_\gamma+ \sum_{\abs{\gamma}=h}\vect{T}^\gamma \phi_\gamma
	\] 
	for every $h\in \N$.
	
	Define $\widetilde m_\gamma\coloneqq \Mc_{\Lc}(\widetilde \phi_\gamma)\in \Sc(\sigma(\Lc))$ and $m_\gamma\coloneqq \Mc_{\Lc_{A'}}(\phi_\gamma)\in C_0(\sigma(\Lc_{A'}))$ for every $\gamma$.
	Then, 
	\[
	m_0(\lambda, \omega') =\sum_{\abs{\gamma}<h} \omega'^\gamma \widetilde m_\gamma(\lambda)+ \sum_{\abs{\gamma}=h} \omega'^\gamma m_\gamma(\lambda, \omega')
	\]
	for every $h\in \N$ and for every $(\lambda, \omega')\in \sigma(\Lc_{A'})$. 
	
	{\bf2.} As in the proof of Theorem~\ref{prop:20:5}, we may reduce to the case in which $\widetilde m_\gamma=0$ for every $\gamma$. 
	Let $L\colon E_{\Lc_A}\to E_{\Lc_{A'}}$ be the unique linear mapping such that $L(\Lc_A)=\Lc_{A'}$, and denote by $L'$ the mapping $\gf_2^*\ni \omega \mapsto (\omega(T_1),\dots, \omega(T_{n_2'}))\in \R^{n_2'}$, so that
	\[
	L(\Set{r}\times r S(\vect{T}))=\Set{r}\times r L'(S)= \left( \Set{r}\times \R^{n_2'}\right) \cap \sigma(\Lc_{A'})
	\]
	for every $r>0$.
	Now, define
	\[
	\widetilde M(\omega)\coloneqq \int_G \phi(x,t) e^{-\frac{1}{4} \abs{x_{\omega}}^2+i \omega(t)  }\,\dd (x,t)
	\]
	for every $\omega\in \gf_2^*$. 
	Arguing as in the proof of Theorem~\ref{prop:20:5} and taking into account Proposition~\ref{prop:11:3}, we see that $\widetilde M\in \Sc(\gf_2^*)$ and that $\widetilde M$ vanishes of order $\infty$ at $0$. 
	Now, observe that
	\[
	m_0( \mi_\omega(\vect{n_{1,\omega}}) ,L'(\omega(\vect{T})))= \widetilde M(\omega) 
	\] 
	for every $\omega\in \gf_2^*$.
	In addition, $\Sigma\coloneqq\R_+(\Set{1}\times S)$ is a closed semianalytic subset of $E_{\Lc_A}$ since it is the closure of the graph of an analytic function (defined on $\gf_2^*\setminus \Set{0}$); in addition, $L$ is proper on $\Sigma$ and $L(\Sigma)=\sigma(\Lc_{A'})$ is a subanalytic closed convex cone, hence Nash subanalytic. 
	By Theorem~\ref{teo:7}, in order to prove that $m_0\in \Sc_{E_{\Lc_{A'}}}(\sigma(\Lc_{A'}))$  it suffices to show that $\widetilde M$ is a formal composite of $L'$. 
	Now, the assertion is clear at $0$ since $\widetilde M$ vanishes of order $\infty$ at $0$.
	Then, take $\omega\in S$. 
	If $\ker L\not \subseteq T_{\omega(\vect{T})} (S(\vect{T}))$, then $L'$ is a submersion at $\omega$, so that the assertion follows in this case.
	Otherwise, as in the proof of Lemma~\ref{lem:6} we see that $L'^{-1}(L'(\omega))=\Set{\omega}$, so that the assertion follows from Corollary~\ref{cor:20:1}.
	By homogeneity, the assertion follows for every $\omega\neq 0$.
	Therefore, $m_0\in \Sc_{E_{\Lc_{A'}}}(\sigma(\Lc_{A'}))$, whence the result.
\end{proof}

\section{Examples: $H$-Type Groups}\label{sec:15}\label{sec:7}

In this section we shall deal with the following situation: $G$ is an $H$-type group and there is a finite family $(\vf_\eta)_{\eta\in H}$ of subspaces of $\gf_1$ such that $\vf_\eta\oplus \gf_2$, with the induced structure, is an $H$-type Lie algebra for every $\eta\in H$ and such that $\vf_{\eta_1}$ and $\vf_{\eta_2}$ commute and are orthogonal for every $\eta_1,\eta_2\in H$ such that $\eta_1\neq \eta_2$.
We shall define $\vect{n_1}\coloneqq \left(\frac{1}{2}\dim \vf_\eta \right)_{\eta\in H}$.

We shall then consider, for every $\eta\in H$, the group of linear isometries $O(\vf_\eta )$ of $\vf_\eta $, and define a canonical action of $O\coloneqq \prod_{\eta\in H} O(\vf_\eta )$ on the \emph{vector space} subjacent to $\gf$ as follows: $(L_\eta )((v_\eta ),t)\coloneqq ((L_\eta \cdot v_\eta ),t)$ for every $(L_\eta )\in O$ and for every $((v_\eta ),t)\in \gf_1\oplus \gf_2$. 

A projector of $\Dc'(G)$ is then canonically defined as follows:
\[
\pi_*(T)\coloneqq \int_{O}  (L\,\cdot\,)_*(T)\,\dd \nu_{O}(L)
\]
for every $T\in \Dc'(G)$; here, $\nu_{O}$ denotes the \emph{normalized} Haar measure on $O$.

\begin{prop}\label{prop:15:1}
	The following hold:
	\begin{enumerate}
		\item $\pi$ induces a continuous projection on $\Dc'^r(G)$, $\Sc'(G)$, $\Ec'^r(G)$, $\Ec^r(G)$, $\Sc(G)$, $\Dc^r(G)$ and $L^p(G)$ for every $r\in \N\cup\Set{\infty}$ and for every $p\in [1,\infty]$;
		
		\item if $\phi_1,\phi_2\in \Dc(G)$, then
		\[
		\langle \pi_*(\phi_1), \phi_2\rangle=\langle \phi_1,\pi_*(\phi_2)\rangle \qquad \text{and} \qquad \langle \pi_*(\phi_1)\vert\phi_2\rangle=\langle \phi_1\vert\pi_*(\phi_2)\rangle;
		\] 
		
		\item if $\mi$ is a positive measure on $G$, then also $\pi_*(\mi)$ is a positive measure; in addition, $\pi_*(\nu_G)=\nu_G$;
		
		\item if $T\in \Dc'(G)$ is $O$-invariant, then also $\check T$ is $O$-invariant;
		
		\item if $T$ is supported at $e$, then $\pi_*(T)$ is supported at $e$; 
		
		\item if $\phi_1,\phi_2\in \Dc(G)$ are $O$-invariant, then also $\phi_1*\phi_2$ is $O$-invariant and $\phi_1*\phi_2=\phi_2*\phi_1$. 
	\end{enumerate}
\end{prop}

The proof is based on~\cite{DamekRicci} and is omitted.

Now, let $\Lc_\eta $ be the differential operator corresponding to  the restriction of the scalar product to $\vf_\eta^* $; in other words, $\Lc_\eta $ is minus the sum of the squares of the elements of any orthonormal basis of $\vf_\eta $. 
Let $T_1,\dots, T_{n_2}$ be an orthonormal basis of $\gf_2$, and define $\Lc_A\coloneqq ( (\Lc_\eta)_{\eta\in H}, (-i T_1,\dots,- i T_{n_2} ))$.

Recall that a left-invariant differential operator $X$ is $\pi$-radial if and only if $\pi_*(X_e)=X_e$, that is, if and only if $X_e$ is $O$-invariant. Nevertheless, this does \emph{not} imply that $X$ is $O$-invariant.

\begin{prop}\label{prop:15:2}
	$\Lc_A$ is a Rockland family and generates (algebraically) the unital algebra of left-invariant differential operators which are $\pi$-radial.
\end{prop}

\begin{proof}
	Since $\sum_{\eta\in H} \Lc_\eta$ is the operator associated with the scalar product of $\gf_1^*$, it is clear that $\Lc_A$ is a Rockland family.
	
	Now, take an $O$-invariant distribution $S$ on $G$ which is supported at $e$. 
	Let $p\colon G\to \quot{G}{[G,G]}$ be the canonical projection. 
	Then,  $p(S)$ is $O$-invariant and supported at $p(e)$.
	By means of the Fourier transform, we see that there is a unique polynomial $P_0\in \R[H]$ such that $p(S)=P_0(p(\Lc_{H,e}))$.
	Therefore, there are $S_1,\dots, S_{n_2}\in \Dc'(G)$ such that $\Supp{S_k}\subseteq \Set{e}$ for every $k=1,\dots,n_2$, and such that
	\[
	S=P_0(\Lc_{H,e})+\sum_{k=1}^{n_2}T_{k,e} S_k.
	\]
	Reasoning by induction, it follows that $S$ belongs to the unital algebra (algebraically) generated by $\Lc_{A,e}$. 
	Conversely, it is clear that $T_1,\dots, T_{n_2}$ are $\pi$-radial. 
	On the other hand, a direct computation shows that $\Lc_{\eta ,e}=-\sum_{v\in B}\partial_v^2$, where $B$ is any orthonormal basis of $\vf_\eta $. Hence, $\Lc_{\eta ,e}$ is $O$-invariant.	
\end{proof}

Now, we shall consider some image families of $\Lc_A$. 
More precisely, we shall fix $\mi\in (\R^H)^{H'}$ so that the induced mapping from $\R^H$ into $\R^{H'}$ is proper on $\R_+^H$. 
Then, we shall define $L\colon E_{\Lc_A} \ni (\lambda_1,\lambda_2)\mapsto (\mi(\lambda_1), \lambda_2)\in \R^{H'}\times \gf_2^*$ and consider the family $L(\Lc_A)$.
Then, $L(\Lc_A)$ is a Rockland family since $L$ is proper on $\sigma(\Lc_A)$ by construction.

\begin{prop}\label{prop:15:4}
	Set $d\coloneqq \dim_\Q \mi(\Q^H)$. 
	Then, there are a $\beta_{L(\Lc_A)}$-measurable function $m\colon E_{L(\Lc_A)}\to \C^{d}$ and a linear mapping $L'\colon \R^{d}\to \C^{H'}$ such that the following hold:
	\begin{itemize}
		\item there is $\mi'\in ((\Q_+^*)^H)^d$ such that $m(\Lc_A)=\mi'(\Lc_H)$;
		
		\item $(L'(m(\Lc_A)),(-i T_j)_{j=1}^{n_2})=L(\Lc_A)$;
		
		\item $m$ equals $\beta_{L(\Lc_A)}$-almost everywhere a continuous function if and only if $d=\dim_\R \mi(\R^H)$.
	\end{itemize} 
\end{prop}

\begin{proof}
	Indeed, we may find $d$ linearly independent $\Q$-linear functionals $p_1,\dots, p_d$ on $\mi(\Q^{H})$. 
	Let $\mi'_1,\dots, \mi'_d$ be the elements of $\Q^{H}$ associated with $p_1\circ \mi, \dots, p_d\circ \mi$. 
	Then, $\mi'_1,\dots, \mi'_d$ are linearly independent over $\Q$, hence over $\C$ by tensorization.
	Now, define $\Lc''_h\coloneqq \sum_{\eta\in H} \mi'_{h,\eta} \Lc_\eta$, so that the family $(\Lc''_1,\dots, \Lc''_d)$ is linearly independent over $\C$. 
	Next, take $h\in \Set{1,\dots, d}$, and observe that, if $\lambda\in \R^{n_2}\setminus \Set{0}$ and $\gamma_1,\gamma_2\in \N^H$ are such that
	\[
	(\abs{\lambda}\mi(\vect{n_1}+2\gamma_1),\lambda)=(\abs{\lambda}\mi(\vect{n_1}+2\gamma_2),\lambda),
	\] 
	then $\mi(\gamma_1-\gamma_2)=0$, so that
	\[
	(\abs{\lambda} \mi'(\vect{n_1}+2\gamma_1),\lambda)=(\abs{\lambda} \mi'(\vect{n_1}+2\gamma_2),\lambda).
	\]
	Hence, there is a $\beta_{\Lc_A}$-measurable function $m\colon E_{L(\Lc_A)}\to \C^d$ such that
	\[
	m_h(L(\lambda'))=\mi'_h (\lambda')
	\]
	for every $\lambda'\in \sigma(\Lc_{A})\cap (\R^H\times (\R^{n_2}\setminus \Set{0}))$; hence, $\Lc''_h\delta_e=\Kc_{\Lc_A}(m_h)$ for every $h=1,\dots, d$.
	Next, observe that, for every $\eta '\in H'$ there is $(L'_{\eta ',1},\dots, L'_{\eta ',d})\in \Q^d$ such that 
	\[
	\sum_{h=1}^d L'_{\eta ',h} (p_h\circ \mi)=\mi_{\eta '}
	\]
	on $\Q^H$. Therefore, $\sum_{h=1}^k L'_{\eta ',h} \mi'_h=\mi_{\eta '}$, whence $(L'(m(\Lc_A)),(-i T_j)_{j=1}^{n_2})=L(\Lc_A)$.
	
	If $d=\dim_\R \mi(\R^H)$, then $m\times \id_{\R^{n_2}}$ is a homeomorphism of $\sigma(L(\Lc_A))$ onto $\sigma(\Lc''_1,\dots, \Lc''_d, (-i T_h)_{h=1}^{n_2})$. 
	Conversely, assume that $m$ can be taken so as to be continuous. 
	Then,  $m\times \id_{\R^{n_2}}$ and $L'\times \id_{\R^{n_2}}$ are inverse of one another between $\sigma(L(\Lc_A))$ and $\sigma(\Lc''_1,\dots, \Lc''_d, (-i T_h)_{h=1}^{n_2})$. 
	In particular,
	$L'$ induces a homeomorphism of $\mi'(\R_+^H)$ onto $\mi(\R_+^H)$, so that these two cones must have the same dimension. Hence, $d=\dim_\R(\mi(\R^H))$.
\end{proof}

\begin{teo}\label{teo:15:1}
	The following conditions are equivalent:
	\begin{enumerate}
		\item[(i)] $\chi_{L(\Lc_A)}$ has a continuous representative;
		
		\item[(ii)] $L(\Lc_A)$ satisfies property $(RL)$;
		
		\item[(iii)] every element of $\Sc_{L(\Lc_A)}(G)$ has a continuous multiplier;
		
		\item[(iv)] $L(\Lc_A)$ satisfies property $(S)$;
		
		\item[(v)] $L(\Lc_A)$ is functionally complete;
		
		\item[(vi)] $\dim_\Q \mi(\Q^H)=\dim_\R \mi(\R^H)$.
	\end{enumerate}
	If, in addition, $L(\Lc_A)$ is not functionally complete, then there is some $L'$, corresponding to some $\mi'\in (\R^H)^{H''}$, such that $L'(\Lc_A)$ is functionally complete and functionally equivalent to $L(\Lc_A)$.
\end{teo}

\begin{proof}
	{\bf (i) $\implies$ (ii).} Obvious.
	
	{\bf (ii) $\implies$ (iii).} Obvious.
	
	{\bf (iii) $\implies$ (vi).} Assume, on the contrary, that $\dim_\Q \mi(\Q^{H})>\dim_\R \mi(\R^{H})$, and keep the notation of Proposition~\ref{prop:15:4}. 
	Then, $m_h$ cannot be taken so as to be continuous for some $h\in \Set{1,\dots, d}$. 
	Take $\phi \in \Sc(E_{L(\Lc_A)})$ so that $\phi(\lambda)\neq 0$ for every $\lambda\in E_{\Lc_A}$.
	Then, 
	\[
	\Kc_{L(\Lc_A)}(m_h \phi)=(\mi'_h(\Lc_H))\Kc_{L(\Lc_A)}(\phi)\in \Sc(G),
	\]
	but $m_h \phi$ is not equal $\beta_{L(\Lc_A)}$-almost everywhere to any continuous functions, whence the result.	
	
	{\bf (vi) $\implies$ (iv).} This follows from Theorem~\ref{prop:20:5}.
	
	{\bf (iv) $\implies$ (v).} This follows from Proposition~\ref{prop:12:1}.
	
	{\bf (v) $\implies$ (vi).} This follows from Proposition~\ref{prop:15:4}.
	
	{\bf (vi) $\implies$ (i).} This follows from~Theorem~\ref{prop:19:4}.
\end{proof}

\section{Examples: Products of Heisenberg Groups}\label{sec:8}

In this section, $(G_\alpha)_{\alpha\in A}$ will be a family of Heisenberg groups each of which is endowed with a homogeneous sub-Laplacian $\Lc_\alpha$. 
Define $\Lc\coloneqq \sum_{\alpha\in A}\Lc_\alpha$, and denote by $\Tc$ a finite family of elements of $\gf_2$, which is the centre of the Lie algebra of $G\coloneqq \prod_{\alpha\in A} G_\alpha$.

Before we proceed to the main results of these section, let us introduce some more notation. 
For every $\alpha\in A$, we shall denote by $T_\alpha$ a basis of the centre of the Lie algebra of $G_\alpha$, so that we may identify canonically $\gf_2$ with $\bigoplus_{\alpha\in A} \R T_{\alpha}$. 
Then, there is a basis $(X_{\alpha,1},\dots, X_{\alpha,2 n_{1,\alpha}}, T_\alpha)$ of the Lie algebra of $G_\alpha$ such that $[X_{\alpha,k},X_{\alpha,n_{1,\alpha}+k}]=T_\alpha$ for every $k=1,\dots, n_{1,\alpha}$, while the other commutators vanish, and such that there is $\mi_\alpha\in (\R_+^*)^{n_{1,\alpha}}$ such that
\[
\Lc_\alpha= -\sum_{k=1}^{n_{1,\alpha}} \mi_{\alpha,k} (X_{\alpha,k}^2+X_{\alpha,n_{1,\alpha}+k}^2).
\]	
We shall denote by $\gf_{1,\alpha}$ the vector space generated by $X_{\alpha,1},\dots, X_{\alpha,2 n_{1,\alpha}}$, and we shall set $\vect{n_1}\coloneqq (n_{1,\alpha})_{\alpha\in A}$.

\begin{prop}
	Assume that $\card(A)\Meg 2$. If $\Tc$ generates $\gf_2$, then the families $(\Lc, - i \Tc)$ and $(\Lc_A, - i \Tc)$ are functionally equivalent. 
	In addition, $(\Lc,-i \Tc)$ does not satisfy properties $(RL)$ and $(S)$.
\end{prop}

\begin{proof}
	See Theorem~\ref{prop:19:4} and its proof.
\end{proof}

\begin{lem}\label{lem:20:2}
	Let $\mi$ be a linear mapping of $\R^n$ onto $\R^m$ such that $\ker \mi \cap \R_+^n=\Set{0}$.
	Define $\Sigma_0\coloneqq \mi(\R_+^{n})\times \Set{0}$, and 
	\[
	\Sigma\coloneqq \Set{(\lambda \mi (\vect{1}_{n}+2 \gamma), \lambda )\colon \lambda>0, \gamma\in \N^{n}}\cup\Sigma_0.
	\]
	If $\phi\in C^\infty(\R^{m}\times \R)$ vanishes on $\Sigma$, then $\phi$ vanishes of order $\infty$ on $\Sigma_0$.
\end{lem}

\begin{proof}
	Take $x=(\lambda \mi(\vect{1}_{n}+2 \gamma),0)$ for some $\lambda>0$ and some $\gamma\in \N^{n}$. 
	Then, for every $k\in\N$,
	\[
	\left(x_1, \frac{\lambda}{2 k+1}\right)= \left( \frac{\lambda}{2 k+1}\mi(  \vect{1}_{n}+ 2( (2k+1)\gamma+k \vect{1}_{n} )) ,\frac{\lambda}{2 k+1}    \right)\in \Sigma.
	\]
	Therefore, it is easily seen that $\partial^h_2 \phi (x)=0$ for every $h\in \N$.  
	Since the set 
	\[
	\Set{(\lambda \mi (\vect{1}_{n}+2 \gamma),0)\colon \lambda>0, \gamma\in \N^n }
	\]
	is dense in $\Sigma_0$, it follows that $\partial^h_2 \phi$ vanishes on $\Sigma_0$ for every $h\in \N$.
	Then, observe that, since we assumed that $ \mi(\R^{n})=\R^{m}$,  the closed convex cone $\Sigma_0$ generates $\R^{m}\times \Set{0}$, so that $\Sigma_0$ is the closure of its interior in $\R^m\times \Set{0}$. The assertion follows easily.
\end{proof}

\begin{teo}\label{prop:13}
	Assume that $\card(A)\Meg 2$. If $\Tc$ does not generate $\gf_2$, then the family $(\Lc, - i \Tc)$ satisfies properties $(RL)$ and $(S)$.
\end{teo}

\begin{proof}
	{\bf1.} Let us prove that $(\Lc,- i \Tc)$ satisfies property $(RL)$. 
	Consider the Rockland family $(\Lc, - i T_A)$ and take $\alpha\in A$; take $\omega\in \R^{A}$. 
	Define 
	\[
	C_\gamma\coloneqq 
	\left\{\left(\sum_{\alpha\in A}\abs{\omega_\alpha} \mi_\alpha (\vect{1}_{n_{1,\alpha}}+2\gamma_\alpha), \omega  \right) \colon \omega \in \R^A \right\}
	\]
	for every $\gamma\in \N^{\vect{n_1}}$, so that $C_0$	is the boundary of a convex polyhedron. 
	If $L\colon E_{(\Lc, - i T_A)}\to E_{(\Lc,-i \Tc)}$ is the unique continuous linear mapping such that $L(\Lc, - i T_A)=(\Lc, - i \Tc)$, then $\chi_{C_0}\cdot \beta_{(\Lc, - i T_A)}$ is $L$-connected by Proposition~\ref{prop:A:8}.
	Now, define $\Lc'_{A'}\coloneqq   (( - X_{\alpha,k}^2-X_{\alpha,n_{1,\alpha}+k}^2  )_{k=1,\dots,n_{1,\alpha}},- i T_\alpha)_{\alpha\in A}$, so that $\Lc'_{A'}$ satisfies properties $(RL)$ and $(S)$ by Theorems~\ref{teo:4} and~\ref{teo:15:1}.
	Take $f\in L^1_{(\Lc,- i \Tc)}(G)$, and let $\widetilde m$ be its continuous multiplier relative to $\Lc'_{A'}$ (cf.~Theorem~\ref{teo:15:1}).
	Then, 
	\[
	m_\gamma\colon C_\gamma\ni \left(\sum_{\alpha\in A}\abs{\omega_\alpha} \mi_\alpha (\vect{1}_{n_{1,\alpha}}+2\gamma_\alpha), \omega  \right) \mapsto \widetilde m( (\abs{\omega_\alpha} (\vect{1}_{n_{1,\alpha}}+2 \gamma_\alpha),\omega_\alpha)_{\alpha\in A})
	\]
	is a continuous function on $C_\gamma$ which equals $\Mc_{(\Lc,- i T_A)}(f)$ $\chi_{C_\gamma}\cdot \beta_{(\Lc, - i T_A)}$-almost everywhere. 
	Therefore, the assertion follows from Theorem~\ref{prop:19:2}.

	{\bf2.} Assume that $\Tc$ generates a hyperplane of $\gf_2$, and let us prove that $(\Lc,- i \Tc)$ satisfies property $(S)$. 
	Take $m\in C_0(E_{(\Lc,-i \Tc)})$ such that $\Kc_{(\Lc, - i \Tc)}(m)\in \Sc(G)$, and 
	consider the (unique) linear mapping
	\[
	L'\colon E_{\Lc'_{A'}}\to E_{(\Lc,- i \Tc)}
	\]   
	such that $L'(\Lc'_{A'})=(\Lc,- i \Tc)$. 
	Then, there is $m_0\in \Sc(E_{\Lc'_{A'}})$ such that $m\circ L'=m_0$ on $\sigma(\Lc'_{A'})$. 
	Next, define, for every $\eps\in \Set{-,+}^A$ and for every $\gamma\in \N^{\vect{n_1}}$,
	\[
	S_{\eps,\gamma}\coloneqq \left\{ \left(\abs{\omega_\alpha}   (\vect{1}_{n_{1,\alpha}}+2 \gamma_\alpha), \omega_\alpha  \right)_{\alpha\in A} \colon \omega \in \prod_{\alpha\in A}\R_{\eps_\alpha}\right\},
	\]
	so that $S_{\eps,\gamma}$ is a closed convex semi-algebraic set of dimension $\card(A)$.
	Assume that $L'$ is not one-to-one on $S_{\eps,\gamma}$. 
	Since $L'$ is proper on $\sigma(\Lc'_{A'})$, for every  $\lambda\in L'(S_{\eps,\gamma})$ the fibre $L'^{-1}(\lambda)$ intersects $S_{\eps,\gamma}$ on a closed segment whose end-points lie in the relative boundary of $S_{\eps,\gamma}$.
	Therefore, $L'(S_{\eps,\gamma})$ gives no contribution to $\bigcup_{\eps'\in \Set{-,+}^A} L'(S_{\eps',\gamma})$; in particular, we may find a subset $E_0$ of $\Set{-,+}^A$ such that $\bigcup_{\eps\in E_0}L'(S_{\eps,0})=L'(\sigma(\Lc'_{A'}))$ and such that $L'$ is one-to-one on $S_{\eps,0}$ for every $\eps\in E_0$.
	
	Now, Corollary~\ref{cor:A:7} implies that for every $\eps\in E_0$ there is $m'_\eps\in \Sc(E_{(\Lc,- i \Tc)})$ such that $m'_\eps\circ L'=m_0$ on $S_{\eps,0}$. 
	Nevertheless, we must prove that these functions  $m'_\eps$ can be patched together to form a Schwartz multiplier of $\Kc_{(\Lc,- i \Tc)}(m)$.
	Then, take  $\lambda\in \sigma(\Lc,- i \Tc)$. 	We shall distinguish some cases.
	
	Assume that there are $\eps_1,\eps_2\in E_0$ such that $L'(S_{\eps_1,0})\cap L'(S_{\eps_2,0})$ has non-empty interior and such that $\lambda\in L'(S_{\eps_1,0})\cap L'(S_{\eps_2,0})$. 
	Then, $m'_{\eps_1}=m'_{\eps_2}$ on $L'(S_{\eps_1,0})\cap L'(S_{\eps_2,0})$, so that $m'_{\eps_1}-m'_{\eps_2}$ vanishes of order infinity on the closure of the interior of $L'(S_{\eps_1,0})\cap L'(S_{\eps_2,0})$, which is $L'(S_{\eps_1,0})\cap L'(S_{\eps_2,0})$ by convexity. 
	In particular, $m'_{\eps_1}-m'_{\eps_2}$ vanishes of order infinity at $\lambda$.
	
	Next, assume that there are $\eps_1,\eps_2\in E_0$ and $\lambda'\in S_{\eps_1,0}\cap S_{\eps_2,0}$ such that $L'(\lambda')=\lambda$. 
	Then, $\lambda'\in S_{\eps_k,\gamma}$ for every $\gamma\in \N^{\vect{n_1}} $ such that $\gamma_\alpha=0$ if $\eps_{1,\alpha}=\eps_{2,\alpha}$, and for $k=1,2$; let $\Gamma_{\eps_1,\eps_2}$ be the set of such $\gamma$. 
	Now, clearly $m'_{\eps_k}\circ L'=m_0$ on $S_{\eps_k,\gamma}$ for every $\gamma\in \Gamma_{\eps_1,\eps_2}$. 
	Taking into account Lemma~\ref{lem:20:2}, we see that the restriction of $(m'_{\eps_1}-m'_{\eps_2})\circ L'$ to $\prod_{\alpha \in A} V_\alpha$ vanishes of order $\infty$ at $\lambda'$, where $V_\alpha$ is $\R(\vect{1}_{n_{1,\alpha}}, \eps_{1,\alpha})$ if $\eps_{1,\alpha}=\eps_{2,\alpha}$ while $V_\alpha= \R^{n_{1,\alpha}+1}$ otherwise. 
	Since either $\eps_1=\eps_2$ or $L'\colon \prod_{\alpha \in A} V_\alpha\to E_{(\Lc,-i \Tc)}$ is onto, it follows that $m'_{\eps_1}-m'_{\eps_2}$ vanishes of order $\infty$ at $\lambda$.
	
	Then, assume that there are $\eps_1,\eps_2\in E_0$ such that $\lambda\in L'(S_{\eps_1,0})\cap L'(S_{\eps_2,0})$, but that  $L'(S_{\eps_1,0})\cap L'(S_{\eps_2,0})$ has empty interior and $\lambda \not \in L'(S_{\eps_1,0}\cap S_{\eps_2,0})$. 
	Let us prove that there is $\eps_3\in E_0$ such that $\lambda\in L'(S_{\eps_1,0}\cap S_{\eps_3,0})$ and such that $L'( S_{\eps_2,0})\cap L'(S_{\eps_3,0})$ has non-empty interior. 
	Indeed, observe that there is a unique liner mapping $L''$ such that $L''(\Lc'_{A'})=(\Lc, - i T_A)$. 
	In addition, if $S_0= \bigcup_{\eps\in \N^A} S_{\eps,0}$, then $L''$ induces a homeomorphism of $S_0$ onto $S_0'=L''(S_0)$. 
	Furthermore, $S_0'$ is the boundary of the convex envelope $C_0'$ of $\sigma(\Lc, - i T_A)$, which is a convex polyhedron, and $\ker L$ has dimension $1$. 
	Next, observe that $L(S_0')=L(C_0')=\sigma(\Lc, - i \Tc)$ and that $L$ is proper on $C_0'$; put an orientation on $\ker L$, fix a linear section $\ell$ of $L$, and define
	\[
	g_+\colon \sigma(\Lc,- i \Tc)\ni \lambda \mapsto \max \Set{t\in \ker L\colon  \ell(\lambda)+t\in S_0'  }
	\] 
	and
	\[
	g_-\colon \sigma(\Lc,- i \Tc)\ni \lambda \mapsto \min \Set{t\in \ker L\colon  \ell(\lambda)+t\in S_0'  }.
	\]
	Then, $g_-$ and $g_+$ are convex and concave, respectively, hence continuous on the interior of $\sigma(\Lc,-i\Tc)$. Observe that the union of the graphs of $g_-$ and $g_+$ is $\bigcup_{\eps\in E_0} L''(S_{\eps,0})$.
	Now, let $E_{0,\pm}$ be the set of $\eps\in E_0$ such that $L''(S_{\eps,0})$ is contained in the graph of $g_\pm$. 
	Observe that $E_0$ is the disjoint union of $E_{0,-}$ and $E_{0,+}$, since $g_-(\lambda)\neq g_+(\lambda)$ for every $\lambda$ in the interior of $\sigma(\Lc_A)$ (cf.~the proof of Proposition~\ref{cor:A:6}).
	Therefore, $\sigma(\Lc_A)=\bigcup_{\eps\in E_{0,\pm}} L'(S_{\eps,0})$; since $L'(S_{\eps,0})$ is closed for every $\eps\in E_0$ and since $E_0$ is finite, this proves that the union of the $L'(S_{\eps,0})$ such that $\eps\in E_{0,\pm}$ and $\lambda\in L'(S_{\eps,0})$ is a neighbourhood of $\lambda$ in $\sigma(\Lc_A)$.
	Next, since $\lambda \not \in L'(S_{\eps_1,0}\cap S_{\eps_2,0})$, we may assume that $\eps_1\in E_{0,+}$ and $\eps_2\in E_{0,-}$. 
	Then, there is $\eps_3\in E_{0,+}$ such that $\lambda\in L'(S_{\eps_3,0})$ and $L'(S_{\eps_2,0})\cap L'(S_{\eps_3,0})$ has non-empty interior, so that $\lambda\in L'(S_{\eps_1,0}\cap S_{\eps_2,0})$.  
	Therefore, the preceding arguments show that  $m'_{\eps_2}-m'_{\eps_1}$ vanishes of order $\infty$ at $\lambda$.
	
	Hence, by means of Theorem~\ref{teo:7} we see that there is $m'\in \Sc(E_{(\Lc,- i \Tc)})$ such that $m'\circ L=m_0$ on $\sigma(\Lc'_{A'})$, so that $m'=m$ on $\sigma(\Lc, - i \Tc)$, whence the result in this case.
	
	{\bf3.} Now, consider the general case, and take $m\in C_0(E_{(\Lc,-i \Tc)})$ such that $\Kc_{(\Lc, - i \Tc)}(m)\in \Sc(G)$. 
	Take a finite subset $\Tc'$ of $\gf_2$ which contains $\Tc$ and generates a hyperplane of $\gf_2$, so that~{\bf2} implies that $(\Lc, - i \Tc')$ satisfies property $(S)$. 
	Observe that $\sigma(\Lc,- i \Tc')$ is a convex semi-algebraic set. Therefore, the assertion follows easily from Corollary~\ref{cor:A:7}.
\end{proof}

\begin{lem}\label{lem:2:1}
	Let $G'$ and $G''$ be two non-trivial homogeneous groups, $\Lc'$ and $\Lc''$ two positive Rockland operators on $G'$ and $G''$, respectively. Then, the operator $\Lc'+\Lc''$ on $G'\times G''$ satisfies property $(S)$.
\end{lem}

\begin{proof}
	Let $\pi$ be the canonical projection of $G'\times G''$ onto its abelianization $G'''$, that is, onto its quotient by the normal subgroup $[G'\times G'', G'\times G'']$. Then, Theorem~\ref{teo:9:3} implies that $\dd \pi(\Lc'+\Lc'')$ satisfies property $(S)$. 
	Now, take $\phi \in \Sc_{\Lc'+\Lc''}(G'\times G'')$. Then, Theorem~\ref{teo:9:3} and~\cite[Theorem 3.2.4]{Martini}, applied to the right quasi-regular representation of $G'\times G''$ in $L^2(G''')$,  imply that $\pi_*(\phi)\in \Sc_{\dd \pi(\Lc'+\Lc'')}(G''')$, so that there is $m\in \Sc(\R)$ such that $\Kc_{\dd \pi(\Lc'+\Lc'')}(m)=\pi_*(\phi)$. 
	Since $\sigma(\Lc'+\Lc'')=\R_+=\sigma(\dd \pi(\Lc'+\Lc''))$, we see that $\phi=\Kc_{\Lc'+\Lc''}(m)$, whence the result.
\end{proof}

\begin{teo}\label{prop:16:2}
	Let $G'$ be a homogeneous group endowed with a positive Rockland operator $\Lc'$ which is homogeneous of degree $2$. 
	Then, the following hold:
	\begin{enumerate}
		\item $(\Lc+\Lc',-i \Tc)$ satisfies property $(RL)$;
		
		\item if $\Lc'$ satisfies property $(S)$, then also $(\Lc+\Lc',-i \Tc)$ satisfies property $(S)$.
	\end{enumerate}
\end{teo}

Notice that we do \emph{not} require that $G'$ is graded, so that the requirement that $\Lc'$ has homogeneous degree $2$ can be  met up to rescaling the dilations of $G'$.
In addition, if $\Lc'$ is not positive, then $(\Lc+\Lc',-i \Tc)$ is \emph{not} a Rockland family, since the mapping $\sigma(\Lc,-i \Tc, \Lc')\ni (\lambda_1,\lambda_2,\lambda_3)\mapsto (\lambda_1+\lambda_3, \lambda_2)$ is not proper.

\begin{proof}
	{\bf1.}  Let us prove that $\Lc_A$ satisfies property $(RL)$. Observe that, if the assertion holds when $\Tc$ generates $\gf_2$, then the assertion follows by means of Propositions~\ref{prop:A:8} and~\ref{prop:A:7}.
	Therefore, we may assume that $\Tc$ is a basis of $\gf_2$.
	
	Define $\Lc'_{A'}\coloneqq   (((-X_1^2-X_{1+n_{1,\alpha}}^2,\dots, -X_{n_{1,\alpha}}^2-X_{2 n_{1,\alpha}}^2) ,- i T_{\alpha})_{\alpha\in A}, \Lc')$,
	and observe that $\Lc'_{A'}$ satisfies property $(RL)$ by Theorems~\ref{teo:4} and~\ref{teo:15:1}. 
	Define
	\[
	S_{0}\coloneqq \left\{ \left(\abs{\omega_\alpha}  \vect{1}_{n_{1,\alpha}}, \omega_\alpha  \right)_{\alpha\in A} \colon \omega \in \R^A\right\},
	\]
	so that $S_{0}$ is a closed semi-algebraic set of dimension $\card(A)$.
	Then, apply Proposition~\ref{prop:A:7} with $\beta=\chi_{S_0\times \R_+}\beta_{\Lc'_{A'}}$, observing that $L\colon S_0\times \R_+\to \sigma(\Lc+\Lc',-i \Tc)$ is a proper bijective mapping, hence a homeomorphism.
	Since $L_*(\beta_{(\Lc+\Lc',-i \Tc)})$ is equivalent to $L_*(\beta)$ thanks to~\cite[Theorem 3.2.22]{Federer},	the assertion follows.

	{\bf2.} Now, assume that $\Lc'$ satisfies property $(S)$, and let us prove that $(\Lc+\Lc',-i \Tc)$ satisfies property $(S)$. 
	Observe that, if we prove that the assertion holds when $\Tc$ generates $\gf_2$, then the general case will follow by means of Corollary~\ref{cor:A:7}.
	Therefore, we shall assume that $\Tc=(T_\alpha)_{\alpha\in A}$.
	
	Observe first that $\Lc'_{A'}$ satisfies property $(S)$ by Theorem~\ref{teo:5}.
	Then, take $m\in C_0(\sigma(\Lc_A))$ such that $\Kc_{\Lc_A}(m)\in \Sc(G\times G')$. It follows that there is $m_1\in \Sc(E_{\Lc'_{A'}})$ such that 
	\[
	m\circ L=m_1
	\]
	on $\sigma(\Lc'_{A'})$. 
	Since $S_0\times \R_+$ is a closed semi-algebraic set, by Theorem~\ref{teo:7} it will suffice to show that the class of $m_1$ in $\Sc(S_0\times \R_+)$ is a formal composite of $L$. 
	Now, this is clear at the points of the form $\left(\sum_{\alpha\in A}\abs{\omega_\alpha}\mi_{\alpha}(\vect{n_{1,\alpha}})+ r, \omega\right)$, where $\omega\in (\R^*)^{A}$ and $r\Meg 0$. 
	Arguing by induction on $\card(A)$ and taking Lemma~\ref{lem:2:1} into account, the assertion follows by means of Lemma~\ref{lem:20:3}.
\end{proof}

As a complement to Theorem~\ref{prop:16:2}, we present the following pathological case.

\begin{prop}\label{prop:4}
	Let $(X,Y,T)$ be a standard basis of $\Hd^1$, and let $\Lc'$ be a positive Rockland operator on a homogeneous group $G$. 
	Assume that $(\Lc')$ satisfies property $(S)$ and that $\Lc'^h$ is homogeneous of degree $2$ for some $h\Meg 2$.
	Then, the Rockland family $(-X^2-Y^2+\Lc'^h, -i T)$ is functionally complete and satisfies property $(RL)$, but does \emph{not} satisfy property $(S)$.
\end{prop}

\begin{proof}
	{\bf1.} Define $\Lc\coloneqq -X^2-Y^2$. 
	Then, Theorem~\ref{prop:16:2} implies that $(\Lc+\Lc'^h, -i T)$ is a Rockland family which satisfies the property $(RL)$. 
	Next, take some $\phi \in \Dc(E_{(\Lc,-i T, \Lc')})$  supported in $\Set{(\lambda'_1,\lambda'_2,\lambda'_3)\colon \lambda'_1<3\abs{\lambda'_2}-\lambda'^h_3} $ and equal to $\pr_3$ on a neighbourhood of $(1,1,0)$. Then, 
	\[
	m\colon (\lambda_1,\lambda_2)\mapsto \phi\left(\abs{\lambda_2}, \lambda_2, \sqrt[h]{\lambda_1-\abs{\lambda_2} }\right)
	\]
	is not equal  to any elements of $\Sc(E_{\Lc_A})$ on $\sigma(\Lc_A)$. On the other hand, $\Kc_{\Lc_A}(m)=\Kc_{(\Lc,-i T, \Lc')}(\phi)\in \Sc(\Hd^1\times \R)$.
	Hence, $\Lc_A$ does not satisfy property $(S)$.
	
	{\bf2.} Now, let us prove that $\Lc_A$ in functionally complete. 
	Take $m\in C(E_{\Lc_A})$ such that $\Kc_{\Lc_A}(m)$ is supported in $\Set{e}$. 
	Notice that we may assume that $m$ is continuous since $\Lc_A $ satisfies property $(RL)$. 
	Projecting onto the quotient by $\Set{0}\times\R$, we see that there is a unique polynomial $P$ on $E_{\Lc_A}$ which coincides with $m$ on $\sigma(\Lc,-i T)$. 
	On the other hand, the family $(\Lc,-i T, \Lc')$ is functionally complete since it satisfies property $(S)$ (cf.~Theorem~\ref{teo:5} and Proposition~\ref{prop:12:1}). 
	Hence, there is a unique polynomial $Q$ on $E_{(\Lc,-i T, \Lc')}$ such that
	\[
	m(\lambda_1+\lambda_3^h, \lambda_2)=Q(\lambda_1,\lambda_2,\lambda_3)
	\]
	for every $(\lambda_1,\lambda_2,\lambda_3)\in\sigma(\Lc,-i T, \Lc')$. Hence, 
	\[
	P(\lambda_1+\lambda_3^h, \lambda_2)=Q(\lambda_1,\lambda_2,\lambda_3)
	\]
	for every $(\lambda_1,\lambda_2,\lambda_3)\in \Set{\left( k_1\abs{r}, r, \sqrt[h]{k_2\abs{r} }  \right)\colon r\in \R, k_1\in 2\N+1, k_2\in 2\N  } $.	
	Now, the closure of this latter set in the Zariski topology is $E_{(\Lc,-i T, \Lc')}$, so that $m=P$ on $\sigma(\Lc_A)$. The assertion follows.
\end{proof}

\begin{small}
	\section*{Acknowledgements}
	I would like to thank professor F.\ Ricci for patience and guidance, as well as for many inspiring discussions and for the numerous suggestions concerning the redaction of this manuscript.	
	I would also like to thank Dr A.\ Martini and L.\ Tolomeo for some discussions concerning their work.
\end{small}

\end{document}